\documentclass[preprint]{imsart}
    
\usepackage[english]{babel}
\usepackage{amssymb}
\usepackage{amsmath} 
\usepackage{pb-diagram}   
\usepackage{epstopdf}                
\usepackage{graphicx}
 \usepackage{epsfig}            
\usepackage{color} 
\usepackage{fancyhdr}  
\usepackage{rotating}
\usepackage{pdflscape}
\usepackage[T1]{fontenc}        
\usepackage{url}                
\usepackage{mathtools}                 
\usepackage{vmargin}    
        
\usepackage{lipsum}
\usepackage{amsfonts}
\usepackage{graphicx}
\usepackage{algorithmic}

\ifpdf
  \DeclareGraphicsExtensions{.eps,.pdf,.png,.jpg}
\else
  \DeclareGraphicsExtensions{.eps}
\fi

\usepackage{enumitem}
\usepackage{amsopn}   
      
\setlist[enumerate]{leftmargin=.5in}
\setlist[itemize]{leftmargin=.5in}

\RequirePackage[OT1]{fontenc}
\RequirePackage{amsthm,amsmath}
\RequirePackage[numbers]{natbib}
\RequirePackage[colorlinks,citecolor=blue,urlcolor=blue]{hyperref}
 
\numberwithin{equation}{section}
\theoremstyle{plain}
\newtheorem{theorem}{Theorem}[section]
\newtheorem{rem}{Remark}[section]
\newtheorem{lemma}{Lemma}[section]
\newtheorem{corollary}{Corollary}[section]
{\bf}{\it} 
\newtheorem{defi}{Definition}[section] 
\newtheorem{prop}{Proposition}[section]
     
\newcommand{\bu}{\bullet}     
\newcommand{\bo}[1]{\mathbf{#1}} 
\newcommand{\indic}{\hbox{1\kern-.24em\hbox{I}}}      

\newcommand{\var}{Var_\mu}     
\newcommand{\esp}{\mathbb{E}}     
\newcommand{\trace}{\text{Tr}}

\newcommand{\y}{y}
\newcommand{\x}{x}
\newcommand{\X}{X}
 
\newcommand{\z}{z} 
 
\newcommand{\R}{\mathbb{R}}
\newcommand{\N}{\mathbb{N}}

\newcommand{\M}{f}
\newcommand{\Set}{\mathcal{F}}      
\newcommand{\tef}{f^{tot}} 
\newcommand{\tief}{f^{sup}}  
\newcommand{\tefg}{h}  
\newcommand{\gtef}{h}             

\newcommand{\T}{T}       
\newcommand{\NN}{n}     
\newcommand{\norme}[1]{\left|\left| #1 \right|\right|_{L^2}}

\newcommand{\normf}[1]{\left|\left| #1 \right|\right|_F}

\newcommand{\D}{\Sigma}

\newcommand{\matleq}{\preceq}  
 
\newcommand{\rr}{t}
\newcommand{\ttt}{r}

\newcommand{\argmin}[1]{\underset{#1}{\operatorname{arg}\,\operatorname{min}}\;}

\begin{document}                 
\begin{frontmatter}
\title{Integral equalities and inequalities: a proxy-measure for multivariate sensitivity analysis}
\runtitle{Derivative-based ANOVA and Poincar\'e' inequalities}
   
\begin{aug}               
\author{\fnms{Matieyendou} \snm{Lamboni}\ead[label=e1]{matieyendou.lamboni@gmail.com, matieyendou.lamboni@univ-guyane.fr}}   

\address{University of Guyane and UMR Espace-Dev-228 (University of Guyane, University of R\'eunion, IRD, University of Montpellier), French Guiana, France\\
\printead{e1}}   
              

\runauthor{M. Lamboni}
\end{aug}  
 
\begin{abstract}   
Weighted Poincar\'e-type and related inequalities provide upper bounds of the variance of  functions. Their applications in sensitivity analysis allow for quickly identifying the active inputs. Although the efficiency in prioritizing inputs depends on the upper bounds, the latter can take higher values, and therefore useless in practice. In this paper, an optimal weighted Poincar\'e-type inequality and gradient-based expression of the variance (integral equality) are studied for a wide class of borel probability measures. For a function $\M : \R \to \R^\NN$ with $\NN \in \N^*$, we show that   
$$          
\var\left(\M\right) =\int_{\Omega \times \Omega}   
\nabla \M \left(\x\right) \nabla \M \left(\x'\right)^\T 
 \frac{F\left(\min(\x,\, \x'\right) - F(\x)F(\x')}{\rho(\x) \rho(\x')}
 \,  d\mu(\x) d\mu(\x') \, ,                                
$$              
and                
$$
\var\left(\M\right) \matleq \frac{1}{2}\int_{\Omega}            
\nabla \M \left(\x\right) \nabla \M \left(\x\right)^\T 
 \frac{F(\x)\left( 1-F(\x)\right)}{\left[\rho(\x)\right]^2}
 \,  d\mu(\x)  \, ,                                
$$      
with $ \var\left(\M\right)=\int_{\Omega} \M \M^\T\, d\mu -\int_{\Omega} \M \, d\mu \int_{\Omega} \M^\T\, d\mu$, $F$ and $\rho$ the distribution and the density functions, respectively. \\             
These results are generalized to cope with multivariate functions by making use of cross-partial derivatives, and they allow for proposing a new proxy-measure for multivariate sensitivity analysis, including Sobol' indices. Finally, the analytical and numerical tests show the relevance of our proxy-measure for identifying important inputs by improving the upper bounds from the Poincar\'e  inequalities.                    
\end{abstract}    

        
   
\end{frontmatter}          

\doublespacing         
\section{Introduction}
            
Multivariate sensitivity analysis (MSA) (\cite{lamboni18a, lamboni11, lamboni09, gamboa14,xiao17}), including variance-based sensitivity analysis (VbSA) (\cite{sobol93,saltelli00,ghanem17,owen13b}), is the standard way of assessing the importance of input factors as well as theirs interactions regarding the variability of the model output(s). The sample-based estimations of generalized sensitivity indices (GSIs) from MSA, including Sobol' indices, often make use of a lot of model runs and have been largely investigated (\cite{lamboni18,lamboni18a,lamboni11,gamboa14,lamboni16b,lamboni16,plischke13,borgonovo14,fruth14,liu06,saltelli10b,saltelli02b,sobol01,owen13}). Recently, the improved, sample-based estimators (\cite{lamboni18a,lamboni18,lamboni16b}) still require a lot of model runs to obtain accurate values of sensitivity indices for models with important interactions among input factors and/or in the presence of skewed or heavy-tailed distributions of input factors.\\     
      
In the presence of important interactions among input factors, efficient and low-cost upper bounds of the total and total-interaction sensitivity indices can be useful for quickly selecting important input factors (i.e., screening inputs). So far, the upper bounds of the total and total-interaction sensitivity indices are based on the derivative global sensitivity measure (DGSM) (\cite{morris91,sobol09,kucherenko09}), which is computationally more attractive than VbSA or MSA (sample-based methods). The upper bound of the total index (resp. total-interaction index), which is a (known) constant times the DGSM index (resp. cross-derivative index), can take higher values especially for total-interaction indices  (\cite{lamboni13,roustant14,kucherenko16,roustant17}). \\    
            
Formally, the upper bounds of sensitivity indices are an application of the Poincar\'e-type inequalities or the weighted Poincar\'e-type inequalities (\cite{roustant17,bobkov99,ricciardi05}). Indeed, these inequalities are used to establish the upper bounds of the total and total-interaction sensitivity indices (\cite{lamboni13,roustant14}). Improving (if possible) the constants in Poincar\'e-type' inequalities in one hand, and the constants and weights in the weighted Poincar\'e or related inequalities in the other hand is essential to expect obtaining an efficient and low-cost screening measure when using the upper bounds. Recently, the authors in \cite{roustant17} improve the constant in Poincar\'e' inequalities for some probability measures. \\
                                      
One way of expecting to obtain the best (known) upper bounds of the variance for a probability measure consists in  expanding the variance of functions as an integral equality involving the derivatives. In this paper,  we propose i) a derivative-based expression of a function $\M:\R^d \to \R^\NN$ and its variance (integral equalities) for a wide class of Borel probability measures, and ii) new and optimal weighted Poincar\'e-type' inequalities. An application of these new inequalities allows for developing a new proxy-measure for MSA, that is, simple formulas that approximate the estimators of generalized sensitivity indices, including Sobol' indices. We use the estimators of the upper bounds of the non-normalized GSIs as the non-normalized proxy-measure for MSA. \\                
      
The paper is organized as follows: in Section \ref{sec:poin}, we recall the definitions of Poincar\'e-type' inequalities using the probability theory and the total differential of a function. In Section  \ref{sec:main}, we provide our main results. First, we provide a theoretical foundation for the derivative-based expression of a function $\M : \R^d \to \R^\NN$, with $d\geq 1$ and $\NN\geq 1$. Second, we derive the expression of the variance of $\M$ using the gradient and  cross-partial derivatives of $\M$, the cumulative distribution functions (CDF), and the probability density functions (PDF) of input factors. Third, we give  new and optimal weighted Poincar\'e-type' inequalities. Section \ref{sec:msa} deals with the application of these new results in  uncertainty quantification such as MSA. First, we recall two definitions of generalized sensitivity indices (from MSA) using the sensitivity functionals and the Frobenius norm; and second, we propose the new proxy-measures of the total and total-interaction generalized sensitivity indices and  Sobol' indices, and third, we illustrate our approach on test cases. We conclude this work in Section \ref{sec:con}.                    
						    					       		                        
\section*{Global notation}              
This section defines the symbols that are going to be used throughout
the paper. For integer $d \in \N^*$ and $j =1,\, \ldots,\, d$, we use $\mu_j$ for a Borel probability measure on $\Omega_j \subseteq \R$
 with density $d\mu_j(x_j)/dx_j=\rho_j(x_j)$ w.r.t. the Lebesgue measure, $X_j$ for a random variable or input factor from $\mu_j$, $x_j$ for a realization or a sample value of $X_j$, and $\bo{\X} :=\{X_j, \, j \in \{1,\, \ldots,\, d\} \}$ for a random vector. We use $\rho_j$ (resp. $F_j$) for the positive and continuous probability density function (PDF) (resp. the cumulative distribution function: CDF) of $\X_j$ with  $j =1,\, \ldots,\, d$.  
We use $\Omega :=\Omega_1\times \Omega_2\times \ldots \times \Omega_d$, $\mu:= \otimes_{j \in \{1,\,\ldots,\, d\}} \mu_j$, $\rho(\bo{\x})$, $F(\bo{\x})$ for a joint support, a Borel probability measure on $\Omega \subseteq \R^d$, a joint PDF and a joint CDF of $\bo{\X}$, respectively. \\
         
For integer $\NN \in \N^*$, $\M :  \R^d \to \R^\NN$ denotes a deterministic function that includes $d$ random variables or input factors $\bo{\X}$ and provides $\NN$ outputs $\M(\bo{\X})$. We use $u \subseteq \{1,2, \ldots, d \}$ for a non-empty subset of $\{1,2, \ldots, d \}$, $\bar{u} := \{1,2, \ldots, d \}\backslash u$ 
and $|u|$ for its cardinality (i.e., the number of elements in $u$).  For a given $u$, we use $\bo{\X}_u := \{\X_j, \, j \in u\}$ for a subset of input factors and $\bo{\X}_{\sim u} := \{\X_j, \, j \in \bar{u}\}$  for the vector containing all input factors, except $\bo{\X}_u$. We have the following partition: $\bo{\X} = (\bo{\X}_u, \bo{\X}_{\sim u})$. We use $\Omega_u :=\boldsymbol{\times}_{j \in u} \Omega_j$, $d\mu (\bo{x}_u):=\prod_{j \in u} d\mu_j (x_j)$;  $\rho_u(\bo{x}_u)$ and $F_u(\bo{x}_u)$ for the PDF and CDF of $\bo{\X}_u$, respectively. \\                        
  
For an $\NN \times \NN$ square matrix $\Sigma :=\left(\sigma_{ij},\, i,\, j \in \{1,\, \ldots,\, \NN\} \right)$, the trace ($\trace(\bu)$) and the Frobenius norm ($\normf{\bu}$)
of $\Sigma$ are defined as follows:         
$$
\trace(\Sigma) :=\sum_{i=1}^\NN \sigma_{ii} \, ,
$$ 
$$    
\normf{\Sigma}^2 :=\trace\left(\Sigma \Sigma^\T \right)\, .          
$$         

         
In what follows, we consider only independent input factors (assumption A1) and Borel measurable and differentiable functions $\M :  \R^d \to \R^\NN$ having finite second moments, that is, 
$$
\esp \left[ \norme{\M(\bo{\X})}^2\right] := \int_\Omega \norme{\M(\bo{x})}^2 \, d\mu(\bo{x}) < +\infty \, , 
$$     
with $\norme{\bu}$ the Euclidean norm on $\R^\NN$.          
                                             
\section{Preliminaries} \label{sec:poin} 
\subsection{Poincar\'e-type' inequalities} 
In this section, we recall the well-known Poincar\'e-type inequalities using the theory of probability, which allow for elementary demonstration of our results. \\

Let $f : \R \to \R$ be a Borel measurable and differentiable function and $\nabla f$ be its gradient, $\mu$ be an absolutely continuous probability measure with respect to the Lebesgue measure on its open support $\Omega$. A Borel probability measure $\mu$ on $\Omega \subseteq \R$ admits the Poincar\'e-type inequality (\cite{bobkov99,roustant17}) if there exists a finite and positive constant $0\leq C(\mu) <+\infty$ such that        
\begin{equation} \label{def:pi}   
\int_\Omega f(x)^2 \, d\mu(x) \leq C(\mu) \times \int_{\Omega} \left|   
\nabla f(x) \right|^2 \,  d\mu  \, ,   
\end{equation}           
with $\int_\Omega f(x) \, d\mu(x) =0$. \\                          
                 
We use $C_{op}(\mu)$ for the best constant value of $C(\mu)$, that is, there exists a function $f_0 : \R \to \R$ such that Equation (\ref{def:pi}) becomes equality with $C(\mu)=C_{op}(\mu)$. \\                           
         
Equation (\ref{def:pi}) is an integral inequality, and it gives an upper bound of the variance of $f$, that is,       
$$       
\var\left(f\right) \leq C(\mu) \times \int_{\Omega} \left|   
\nabla f \right|^2 \,  d\mu  \, ,      
$$         
 where $f$ and $\nabla f$ are square-integrable functions.  \\           
To find the best possible upper-bound of the variance, we can introduce weight functions in the right-hand term of Equation (\ref{def:pi}), and this leads to what we call the weighted Poincar\'e-type inequalities. For a Borel measurable weight function $w : \R \to \R_+$, the Borel probability measure $\mu$ admits the weighted Poincar\'e-type inequality (\cite{bobkov09, bobkov09a}) if there exists a finite and positive constant $0\leq C(\mu,\, w) <+\infty$ such that         
     
\begin{equation} \label{def:wpi}          
\var\left(f\right) \leq C(\mu,\, w) \times \int_{\Omega} \left|   
\nabla f \right|^2 w^2 \,  d\mu  \, .  
\end{equation}                  
      
We use $C_{op}(\mu,\,w)$ for the best constant value of $C(\mu, \, w)$, that is, there exists a function $f_0$ such that Equation (\ref{def:wpi}) becomes equality. Equation (\ref{def:pi}) is a particular case of Equation (\ref{def:wpi}) by choosing $w=1$. So far, the choice of $w$ in order to obtain the best constant $C_{op}(\mu,\,w)$ is not obvious, and it can be hard in practice. Although Equations (\ref{def:pi}) and (\ref{def:wpi}) are not directly comparable (\cite{bobkov09}), one motivation behind using the weighted Poincar\'e inequalities is that one can expect to obtain the lowest upper bound of the variance by properly choosing $w$. Of course, the minimum values of the upper bounds of the variance in (\ref{def:pi}) and (\ref{def:wpi}) should be used as the best  known upper bound. \\         
     
For a function $\M : \R \to \R^\NN$, Equation (\ref{def:wpi}) becomes 
\begin{equation} \label{def:wpim}         
\var\left(\M\right) \matleq C(\mu,\, w)  \int_{\Omega}              
\nabla \M  \nabla \M^\T w  \,  d\mu  \, ,     
\end{equation}        
where  $ \var\left(\M\right)=\int_{\Omega} \M \M^\T\, d\mu -\int_{\Omega} \M \, d\mu \int_{\Omega} \M^\T\, d\mu$ and $\matleq$ is the symbol for the Loewner partial order on matrices, that is, $\Sigma_1 \matleq \Sigma_2$ if $\Sigma_2 -\Sigma_1$ is a positive semi-definite matrix.
 
\subsection{Total differential of a function} \label{sec:tdf}
This section recalls the total differential of a function $\M : \R^d \to \R^\NN$, as our results make use of that decomposition. \\     

Let $\bo{X}$ be $d$ input factors and $\bo{x}$ be a sample value of $\bo{X}$. For a Borel measurable and differentiable function $\M : \R^d \to \R^\NN$, the usual total differential of $\M$ with higher-order terms (i.e., complete total differential of $\M$) at a sample value $\bo{x}$ is the quantity $d\M (\bo{x})$ given by (\cite{courant36,kubicek15})                    
\begin{eqnarray} \label{eq:tde}          
d\M (\bo{x}) & = & \sum_{j=1}^d  \frac{\partial \M}{\partial \x_j}(\bo{x})\, d\x_j
 +    \sum_{1\leq i <j\leq d}^d \frac{\partial^{2} \M}{\partial \bo{\x}_{\{i,j\}}}(\bo{x})  \, d\bo{\x}_{\{i,j\}}
+ \ldots         
+  \frac{\partial^d \M}{\partial \bo{\x}}(\bo{x}) \, d\bo{\x} \, \nonumber \\
  & = & \sum_{\substack{v,\,v \subseteq \{1,\, \ldots,\, d\} \\ |v|>0}} \frac{\partial^{|v|} \M}{\partial \bo{\x}_{v}}(\bo{x})\,  d\bo{\x}_{v} \, ,  
\end{eqnarray}     
where $\frac{\partial^{|v|} \M}{\partial \bo{\x}_{v}} :=\prod_{j \in v} \frac{\partial \M}{\partial x_{j}}$ stands for the $|v|^{\text{th}}$  cross-partial derivatives of each component of $\M$ with respect to each $\x_{j}$, with $j \in v$.  \\                     
                   
Equation (\ref{eq:tde}) expands the total differential of $\M$ as a sum of increasing cross-partial derivatives, which can be treated independently. It shows that the total infinitesimal variation of a function comes from the contribution of all infinitesimal variations of that function with respect to each input and their interactions. One more interesting aspect of this decomposition is that 
$$ \forall\, v\subseteq u, \quad   
 \frac{\partial^{|v|} \M}{\partial \bo{\x}_{v}}(\bo{x})=\bo{0} 
\Longrightarrow 
\frac{\partial^{|u|} \M}{\partial \bo{\x}_{u}}(\bo{x}) =  \bo{0} \, .   
$$ 
 It means that any high-order cross-partial derivative w.r.t. $\bo{x}_u$ vanishes when a low-order cross-partial derivative vanishes. \\  
              
The integral form of the total differential $d\M$ is obtained by integrating Equation (\ref{eq:tde}), and it gives the total variation of $\M$. If we use      
 $u \subseteq \{1,\, \ldots,\, d\}$ for a non-empty subset, $\bo{\x}_{u}$, \, $\bo{\y}_{u}$ and $\bo{\z}_{u}$ for three independent values of $\bo{X}_u$, and $\bo{\x}_{\sim u}$ for a sample value of $\bo{X}_{\sim u}$, the total variation of $\M$ is given as follows (\cite{courant36,kubicek15}):  
      
\begin{eqnarray}   \label{eq:ited}                   
 \M(\bo{\z}_{u}, \bo{\x}_{\sim u}) - \M(\bo{\y}_{u}, \bo{\x}_{\sim u}) & = & \sum_{j \in u} \int_{\y_{j}}^{\z_{j}} 
\frac{\partial \M}{\partial \x_{j}}(\bo{x}) \, d\x_{j}    
 + \sum_{\substack{v,\, v \subseteq u \\ |v|>1}} \int_{\bo{\y}_{v}}^{\bo{\z}_{v}}  
 \frac{\partial^{|v|} \M}{\partial \bo{\x}_{v}}(\bo{x}) \, d\bo{\x}_{v} \, . 
\end{eqnarray}          
                              
\section{Main results} \label{sec:main}
This section aims to provide the decomposition of a function $\M : \R^d \to \R^\NN$ and its variance (i.e., integral equality) in one hand and a new upper bound of the variance of $\M$ in the other hand by making use of derivatives, CDFs and PDFs.\\
     
Namely, let $\M: \R^d \to \R^\NN$ be a Borel measurable function, $\bo{X}$ be $d$ input factors, $\bo{x}$ be a sample value of $\bo{X}$, and $v \subseteq \{1,\, \ldots,\, d\}$ be a non-empty subset (i.e., $|v|>0$). We use $\frac{\partial^{|v|} \M}{\partial \bo{\x}_{v}}(\bo{x}) :=\prod_{j \in v} \frac{\partial \M}{\partial x_{j}}$ for the $|v|^{\text{th}}$ cross-partial derivatives of each component of $\M$ with respect to each $\x_{j}$ with $j \in v$, and $\indic_{[a \geq b]}$ for the indicator function, that is, 
$\indic_{[a \geq b]} =1$ if $a \geq b$ and $0$ otherwise. To derive our results, we assume, in what follows, that  \\           
    
(A1): the $d$ random variables or input factors $\bo{X}$ are independent, \\         

 (A2): the outputs $\M(\bo{X})$ have finite second moments, that is, $\esp\left[\norme{\M(\bo{\X})}^2\right]  < +\infty$, \\ 
    
 (A3): each component of $\M$ is  Borel measurable and differentiable w.r.t. $\bo{x}_v$, $\forall\, v \subseteq \{1,\, \ldots,\, d\}$,\\   
     
 (A4): for all $v \subseteq \{1,\, \ldots,\, d\}$, $\frac{\partial^{|v|} \M}{\partial \bo{\x}_{v}}$ is a measurable function and $\esp\left[\norme{\frac{\partial^{|v|} \M}{\partial \bo{\x}_{v}}(\bo{\X})}^2\right]  < +\infty$, \\ 
       
(A5): a Borel probability measure $\mu_j(dx_j)$ is absolutely continuous w.r.t. the Lebesgue measure with a continuous PDF $\rho_j$ on its open support $\Omega_j \subseteq \R$ and $\rho_j(x_j) \in \R^*_+$ for all $j \in \{1,\, \ldots,\, d\}$.
                  
\subsection{ANOVA-type decomposition based on derivatives} \label{sec:decf}
To establish a new upper bound of the variance of $\M$ in Section \ref{sec:wpc}, we need a decomposition of the variance of $\M$ that involves the derivatives of $\M$. 
This section provides new decompositions of a function $\M : \R^d \to \R^\NN$ and its variance using the gradient and cross-partial derivatives of $\M$, the CDFs and PDFs of input factors.                      
Now, we are going to give in Theorem  \ref{lem:deriaov} the first-type of this decomposition.           
            
\begin{theorem}     \label{lem:deriaov}
Let $\M: \R^d \to \R^\NN$ denote a deterministic function, $\bo{X}$ be $d$ input factors, $F_j$ (resp. $\rho_j$) be the CDF (resp. the PDF) of $X_j$ with $j\in \{1,\, \ldots,\, d\}$, $\bo{x}$ be a sample value of $\bo{X}$, and $u \subseteq \{1,\, \ldots,\, d\}$ be a subset with $|u|>0$. If assumptions (A1)-(A5) hold then \\ 
                   
 $\quad$ (i) we have the following decompositions of $\M$   
   
\begin{equation}   \label{eq:deritef}                            
\M(\bo{\x}) = \int_{\Omega_u} \M(\bo{\x}_{u}, \bo{\x}_{\sim u}) \, d\mu(\bo{\x}_{u})  + \sum_{\substack{v,\, v \subseteq u \\ |v|>0}}
\gtef_v(\bo{\x})  \, , 
\end{equation}               
and       
\begin{eqnarray}   \label{eq:deriaov}                                
\M(\bo{\x})  & = & \M_{\emptyset} +  \sum_{\substack{v,\, v \subseteq \{1,\, 2,\, \ldots,\, d\} \\ |v|>0}}     
\gtef_v(\bo{\x})  \, , 
\end{eqnarray}             
where 
$$ 
\M_{\emptyset} =\int_{\Omega} \M(\bo{x}) \, d\mu(\bo{x})\, , \;  
\gtef_v(\bo{\x}) = \int_{\Omega_v} 
 \frac{\partial^{|v|} \M}{\partial \bo{\x}_{v}}\left(\bo{\x}'_{v},\bo{\x}_{\sim v}\right)  \prod_{j \in v} \frac{F_{j}(\x_{j}') - \indic_{[\x_{j}' \geq \x_{j}]}}{\rho_{j} (\x_{j}') }  \, d\mu(\bo{\x}'_v)\, ,
$$  
with          
$$ \frac{\partial^{|v|}}{\partial \bo{\x}_{v}} :=\prod_{j \in v} \frac{\partial }{\partial x_{j}} \, \text{and} \,  \indic_{[\x_{j}' \geq \x_{j}]} :=\left\{ \begin{array}{cc} 1 & \text{if} \, \x_{j}' \geq \x_{j} \\ 0 & \text{otherwise} \end{array} \right. \, .      
$$                   
         
$\quad$ (ii) The components  
$\gtef_v(\bo{\x}),\, v \subseteq \{1,\, 2,\, \ldots,\, d\}$ with $|v|>0$ in (\ref{eq:deritef}-\ref{eq:deriaov}) satisfy     
\begin{equation} \label{eq:center}        
\int_{\Omega_j}  \gtef_v(\bo{\x}) \, d\mu_j(x_j) =\bo{0} \; \forall \; j\in v \, .   
\end{equation}                      

$\quad$ (iii) The variance of $\gtef_v$, that is, $\var(\gtef_v)$ is given by  
\begin{eqnarray}   \label{eq:decvtief}                                         
\var(\gtef_v)  & = &             
\int_{\Omega\times \Omega_v} 
\frac{\partial^{|v|} \M}{\partial \bo{\x}_{v}}\left(\bo{\x}\right) 
\frac{\partial^{|v|} \M^\T}{\partial \bo{\x}_{v}}\left(\bo{\x}'_{v} ,\bo{\x}_{\sim v}\right) 
\prod_{k \in v} \frac{F_{k}\left[\min(\x_{k},\,\x_{k}')\right] -F_{k}(\x_{k})F_{k}(\x_{k}')}{\rho_{k} (\x_{k})\rho_{k} (\x_{k}')}  \, \nonumber \\ 
 &  &    \times d\mu(\bo{\x}) d\mu(\bo{\x}_v')  \,  .
\end{eqnarray}                  

$\qquad (iv)$ The variance of $\M$, that is, $\var(\M)$ is given by           
\begin{eqnarray}   \label{eq:decvm}                                    
\var(\M)  & = &  
\sum_{\substack{v,\, v \subseteq \{1,\,2,\, \ldots,\, d\} \\ |v|>0}}  \var(\gtef_v)    
  +           
\sum_{\substack{v,\, v \subseteq \{1,\,\ldots,\, d\} \\ |v|>0}} 
\sum_{\substack{w,\, w \subseteq \{1,\,\ldots,\, d\} \\ w \neq v \\ |w|>0}} \int_{\Omega \times \Omega_v\times \Omega_w} 
\frac{\partial^{|v|} \M}{\partial \bo{\x}_{v}}\left(\bo{\x}'_{v},\bo{\x}_{\sim v}\right) 
\frac{\partial^{|w|} \M^\T}{\partial \bo{\x}_{w}}\left(\bo{\z}'_{w},\bo{\x}_{\sim w}\right)   \nonumber \\    
& & \times   
\prod_{j \in v} \frac{F_{j}(\x_{j}') - \indic_{[\x_{j}' \geq \x_{j}]}}{\rho_{j} (\x_{j}') }
 \prod_{k \in w} \frac{F_{k}(\z_{k}') - \indic_{[\z_{k}' \geq \x_{k}]}}{\rho_{k} (\z_{k}') }  \, d\mu(\bo{\x}'_v)d\mu(\bo{\z}'_w)d\mu(\bo{\x})   
 \, . \nonumber \\                            
\end{eqnarray}       
\end{theorem}                                              
\begin{proof}         
See Appendix \ref{app:deriaov}.                                                     
\end{proof}                   
                                  
\begin{rem} On assumption (A3). \\ 
 The derivative-based expressions of $\M$ in Theorem  \ref{lem:deriaov} are still suitable for functions that are continuous but differentiable almost everywhere w.r.t. $\bo{x}_v$ with $v \subseteq \{1,\, \ldots,\ d\}$ .       
\end{rem}         

Theorem \ref{lem:deriaov} gives a full decomposition of a function $\M$ by making use of the gradient and/or cross-partial derivatives of $\M$, CDFs, and PDFs. The decomposition in (\ref{eq:deritef}-\ref{eq:deriaov}) expands the function $\M$  as a sum of components $\gtef_v,\, v \subseteq \{1,\, \ldots,\ d\}$ that allows for assessing the total-interaction effect of $\bo{\X}_v$ (see Section \ref{sec:msa}). Thus, the component $\gtef_v$ in (\ref{eq:deritef}-\ref{eq:deriaov}) can be directly used to assess  either  the overall contribution of $\X_j$ or the overall contribution of interaction between $\bo{\X}_v$, with $|v|>1$ and other inputs over the whole outputs. While the decomposition of $\M$ in Theorem \ref{lem:deriaov} focuses on the total-interaction effect functionals, Theorem \ref{theo:deriaov1} gives the functional ANOVA-type decomposition of $\M$, which allows for easily quantifying the single contribution of input factors and interactions.                       
           
\begin{theorem}        \label{theo:deriaov1}
Let $\M: \R^d \to \R^\NN$ denote a  deterministic function, $\bo{X}$ be $d$ input factors, $F_j$ (resp. $\rho_j$) be the CDF (resp. the PDF) of $X_j$ with $j\in \{1,\, \ldots,\, d\}$, and $\bo{x}$ be a sample value of $\bo{X}$. If assumptions (A1)-(A5)  hold then the derivative-based functional ANOVA of $\M$ is given as follows.\\          
 
$\qquad (i)$ The decomposition of $\M$ is given by  
\begin{equation}      \label{eq:dbfanova} 
 \M(\bo{\x}) =   \M_\emptyset + \sum_{j=1}^d \M_j(\x_j)+
\sum_{\substack{u,\,u \subseteq \{1,2, \ldots d\}\\|u|>1}} \M_u(\bo{\x}_u)\, , 
\end{equation} 
where              
$$      
\M_{\emptyset} =\int_{\Omega} \M(\bo{x}) \, d\mu(\bo{x})\, , 
\qquad 
\M_j(\x_j) = \int_{\Omega} \frac{\partial \M}{\partial \x_{j}}\left(\bo{\x}'\right) 
 \frac{F_{j}(\x_{j}') - \indic_{[\x_{j}' \geq \x_{j}]}}{\rho_{j} (\x_{j}') }  \, d\mu(\bo{\x}') \, ,  
$$
and   
$$              
\M_u(\bo{\x}_u)  =  \sum_{w,\, w\subseteq u} \sum_{\substack{v,\, v \subseteq w \\|v|>0}} (-1)^{|u|-|w|}  \int_{\Omega}   
 \frac{\partial^{|v|} \M}{\partial \bo{\x}_{v}}\left(\bo{\x}'_{v}, \bo{\x}_{w \backslash v}, \bo{\x}_{\sim w}'\right)  \prod_{j \in v} \frac{F_{j}(\x_{j}') - \indic_{[\x_{j}' \geq \x_{j}]}}{\rho_{j} (\x_{j}') }  \, d\mu(\bo{\x}') \, . 
$$        
$\qquad (ii)$ The elements $\M_u(\bo{\x}_u),\, u \subseteq \{1,\,2,\,\ldots,\, d\}$ with $|u|> 0$ are centered and mutually orthogonal, that is, 
\begin{equation}
\int_{\Omega} \M_u(\bo{\x}_u) \,  d\mu_j(\x_j)  =\bo{0} \, \forall \,  j \in u \, , 
\end{equation}        
and for $u_1,\, u_2 \subseteq \{1,\, \ldots, \, d\}$ such that $u_1 \neq u_2$,
\begin{equation} 
\int_{\Omega} \M_{u_1}(\bo{\x}_{u_1}) \M_{u_2}^\T(\bo{\x}_{u_2}) \,   d\mu(\bo{\x}) =\mathcal{O} \, ,             
\end{equation}    
with $\mathcal{O}$ a null matrix.     \\    

$\qquad (iii)$ The decomposition of the variance of $\M$ is given by
\begin{equation} \label{eq:vardec}
\var(\M) = \sum_{\substack{u, \, u \subseteq \{1,\, \ldots,\, d\}\\ |u|>0}} \var(\M_u)\, . 
\end{equation}    
\end{theorem}    
\begin{proof}            
 See  Appendix \ref{app:deriaov1}.                  
\end{proof}             
 
\begin{rem} \label{rem:var}  
Theorems  \ref{lem:deriaov}-\ref{theo:deriaov1} give an integral equality between the variance of $\M$ and a sum of the variances of increasing dimension functions that involve the cross-partial derivatives of $\M$. For a function $\M : \R \to \R^\NN$, both decompositions are equal to 
\begin{equation}  
\M(\x) = \M_\emptyset + \tefg_x(\x)\, ,  
\end{equation}   
with  
$$ \tefg_x(\x) =  \int_{\Omega} \frac{\partial \M}{\partial \x}\left(\x'\right) 
 \frac{F(\x') - \indic_{[\x' \geq \x]}}{\rho_{j} (\x') }  \, d\mu(\x') \, , 
$$           
  and        
\begin{equation} \label{eq:seqvar} 
 \var\left(\M\right) =\int_{\Omega \times \Omega}  
\frac{\partial \M}{\partial \x} \left(\x\right) \frac{\partial \M^\T}{\partial \x} \left(\x'\right)  
 \frac{F\left(\min(\x,\, \x'\right) - F(\x)F(\x')}{\rho(\x) \rho(\x')}
 \,  d\mu(\x) d\mu(\x') \, .  
\end{equation}    
 
It is to be noted that for a function $\M : \R \to \R^\NN$, we have 
$\var(\M) = \var(\tefg_x) $ and $\frac{\partial \M}{\partial \x} \left(\x\right) =\frac{\partial \tefg_x}{\partial \x} \left(\x\right)$.
\end{rem}                    
   
\subsection{Integral equality and new weighted Poincar\'e-tyoe  inequalities} \label{sec:wpc} 
First, this section aims to extend the simple form of the integral equality in  (\ref{eq:seqvar}) to a given class of functions defined on $\R^d$ with $ d\geq 1$ such as the total-interaction effect function used in sensitivity analysis. It establishes an equality relationship between the variance of such functions and the integral of the square of their cross-partial derivatives. Second, we provide new Poincar\'e-type  inequalities, and third, we give new integral inequalities for any $d$-dimension function.    
 
\begin{defi} \label{defi:tief}    
Let $\Set := \left\{\M : \R^d \to \R^\NN : \M \, \text{satisfies assumptions (A1)-(A5)} \, \right\}$  denote the class of all functions satisfying assumptions (A1)-(A5). For a given $\M \in \Set$ and $\forall \, u \subseteq \{1,\, \ldots,\, d\}$ with  $|u|>0$, we define
 
$\quad$ (i)  the total-effect functional by (\cite{lamboni16,lamboni18})
\begin{equation}  \label{eq:deftef}      
\tef_u(\bo{\x}) :=\M(\bo{\x})- \int_{\Omega_u} \M(\bo{\x}) \, d\mu(\bo{x}_u) \, ;          
\end{equation}            
             
$\quad$ (ii) the class of functions $\mathcal{T}^{u, d}$ as follows: 

\begin{equation}  \label{eq:deftiefc0}                     
\mathcal{T}^{u,d} :=         
\left\{ \begin{array}{l}    
\tefg : \R^d \to \R^\NN \, \text{given by} \, \tefg (\bo{\x}) = 
 \pm \int_{\Omega_u} \frac{\partial^{|u|} \M}{\partial \bo{\x}_{u}}\left(\bo{\x}'_{u},\bo{\x}_{\sim u}\right)  \prod_{j \in u} \frac{F_{j}(\x_{j}') - \indic_{[\x_{j}' \geq \x_{j}]}}{\rho_{j} (\x_{j}') }  \, d\mu(\bo{\x}'_{u})
+ \boldsymbol{\alpha} : \\
 \M \in \Set , \, \boldsymbol{\alpha} \in \R^\NN \, \text{and}\, \NN \in \N^*  
\end{array}   \right\}   \, .            
\end{equation}                           
\end{defi}                   
                          
From Definition \ref{defi:tief}, we can derive the following properties of $\tef_u$ and $\mathcal{T}^{u,d}$. First, the functional $\tef_u : \R^d \to \R^\NN$ in (\ref{eq:deftef}) is differentiable w.r.t. $\bo{x}_v, \, \forall \, v \subseteq u$ and we have       
\begin{equation} \label{eq:proptef}  
\frac{\partial^{|v|}\tef_u}{\partial \bo{\x}_{v}}\left(\bo{\x}\right) =
\frac{\partial^{|v|}\M}{\partial \bo{\x}_{v}}\left(\bo{\x}\right) \; \forall \, v \subseteq u \,  ,   
\end{equation}        
as the second term of $\tef_u$ is a function of $\bo{x}_{\sim u}$, only.  We also have $\int_\Omega \tef_u \, d\mu =\bo{0}$. \\     
      
Second, it is to be noted that the class of functions $\mathcal{T}^{u, d}$ is not empty. For instance, let consider a function $\M_0 \in \Set$ given by $\M_0(\bo{x}) = \prod_{j=1}^d \left[\bo{a}_j F_j(x_j) + \bo{b}_j \right]$ with $\bo{a}_j,\, \bo{b}_j \in \R^\NN$. We can see that $\mathcal{T}^{u, d}$ contains a function $\tefg_0 : \R^d \to \R^\NN$ given by    
\begin{equation} \label{eq:ex1}  
\tefg_0(\bo{x}) = \prod_{j \in \bar{u}} \left[\bo{a}_j F_j(x_j) + \bo{b}_j \right] \prod_{j \in u} \left[\bo{a}_j F_j(x_j) - \frac{\bo{a}_j}{2} \right] \, .               
\end{equation} 
   
Also, Lemma \ref{lem:ex} shows that a function $\M_u^{sup} : \R^d \to \R^\NN$ given by 
\begin{equation} \label{eq:ex2}   
\M_u^{sup} (\bo{\x})=                     
 \sum_{\substack{v,\, v \subset u}}  
 (-1)^{|u|-|v| +1} \tef_{u \backslash v}(\bo{\x}) +\boldsymbol{\alpha} \, , 
\end{equation} 
belongs to $\mathcal{T}^{u, d}$ by deriving another expression of $\M_u^{sup}$. 
      
\begin{lemma} \label{lem:ex}
 Let $\M \in \Set$, $\bo{X}$ be $d$ input factors, $F_j$ (resp. $\rho_j$) be the CDF (resp. the PDF) of $X_j$ with $j\in \{1,\, \ldots,\, d\}$, and $\bo{x}$ be a sample value of $\bo{X}$. If assumptions (A1)-(A5)  hold then \\ 

$\quad$ (i) the function defined in (\ref{eq:ex2}) (i.e., $\M_u^{sup}$) is also given by
\begin{eqnarray}  \label{eq:tief10}  
 \M_u^{sup}(\bo{\x}) &=& (-1)^{|u|+1} \int_{\Omega_u}
\frac{\partial^{|u|} \M}{\partial \bo{\x}_{u}}\left(\bo{\x}'_{u},\bo{\x}_{\sim u}\right) 
 \prod_{j \in u} \frac{F_{j}(\x_{j}') - \indic_{[\x_{j}' \geq \x_{j}]}}{\rho_{j} (\x_{j}') }  \, d\mu(\bo{\x}'_{u}) + \boldsymbol{\alpha}  \, .                   
\end{eqnarray}       

$\quad$ (ii) The function $\M_u^{sup}(\bo{\x})$ is differentiable w.r.t. $\bo{\x}_u$, and we have the following relationships: 
                      
\begin{equation}     \label{eq:tieffderiv}    
\frac{\partial^{|u|}\M_u^{sup}}{\partial \bo{\x}_{u}}\left(\bo{\x}\right)
=\frac{\partial^{|u|}\M}{\partial \bo{\x}_{u}}\left(\bo{\x}\right) \, , 
\end{equation}     	 
 and                               
\begin{eqnarray}  \label{eq:tiefder}  
 \M_u^{sup}(\bo{\x}) &=& (-1)^{|u|+1} \int_{\Omega_u}
\frac{\partial^{|u|} \M_u^{sup}}{\partial \bo{\x}_{u}}\left(\bo{\x}'_{u},\bo{\x}_{\sim u}\right) 
 \prod_{j \in u} \frac{F_{j}(\x_{j}') - \indic_{[\x_{j}' \geq \x_{j}]}}{\rho_{j} (\x_{j}') }  \, d\mu(\bo{\x}'_{u}) + \boldsymbol{\alpha}  \, .                   
\end{eqnarray}                              
\end{lemma} 
\begin{proof} 
See Appendix \ref{app:ex2}.                
\end{proof}   

Moreover, we can see that $\forall\, \tefg_1,\, \tefg_2 \in \mathcal{T}^{u, d}$ and $\lambda \in \R$, we have the following properties:
 $$
 \int_\Omega \tefg_1 \, d\mu =\boldsymbol{\alpha} \, , 
$$
and 
$$ \tefg_1 +\lambda \tefg_2 \in \mathcal{T}^{u, d} \, .
$$
    
To derive the integral equality and inequality, it is important to note that the class of function $\mathcal{T}^{u, d}$ contains also functions that share the same derivatives with $\M$, that is, $ \exists \, \tefg \in \mathcal{T}^{u, d}$ such that
$$
\frac{\partial^{|u|}\tefg}{\partial \bo{\x}_{u}}\left(\bo{\x}\right)
=\frac{\partial^{|u|}\M}{\partial \bo{\x}_{u}}\left(\bo{\x}\right) \, . 
$$ 
In what follows, we use assumption (A6) for functions $h\in \mathcal{T}^{u, d}$ satisfying $\frac{\partial^{|u|}\tefg}{\partial \bo{\x}_{u}}\left(\bo{\x}\right)
=\frac{\partial^{|u|}\M}{\partial \bo{\x}_{u}}\left(\bo{\x}\right)$, and 
 Corollary \ref{coro:equa} gives some integral equalities.          
                       
\begin{corollary}        \label{coro:equa}
Let $\M \in \Set$ and $\tefg \in \mathcal{T}^{u, d}$ denote two functions, $\bo{X}$ be $d$ input factors, $F_j$ (resp. $\rho_j$) be the CDF (resp. the PDF) of $X_j$ with $j\in \{1,\, \ldots,\, d\}$, and $\bo{x}$ be a sample value of $\bo{X}$. If assumptions (A1)-(A5)  hold then \\ 

$\quad$ (i) for all  $\tefg \in \mathcal{T}^{u, d}$, we have 
 
\begin{eqnarray}   \label{eq:decvtief0f}                                         
\var(\tefg)  & = &             
\int_{\Omega\times \Omega_u} 
\frac{\partial^{|u|} \M}{\partial \bo{\x}_{u}}\left(\bo{\x}\right) 
\frac{\partial^{|u|} \M^\T}{\partial \bo{\x}_{u}}\left(\bo{\x}'_{u} ,\bo{\x}_{\sim u}\right) 
\prod_{k \in u} \frac{F_{k}\left[\min(\x_{k},\,\x_{k}')\right] -F_{k}(\x_{k})F_{k}(\x_{k}')}{\rho_{k} (\x_{k})\rho_{k} (\x_{k}')}  \, \nonumber \\ 
 &  &    \times         
d\mu(\bo{\x}) d\mu(\bo{\x}_u')          
\,  .    
\end{eqnarray}   

$\quad$ (ii) Moreover, if assumption (A6) holds (i.e., $  
\frac{\partial^{|u|}\tefg}{\partial \bo{\x}_{u}}\left(\bo{\x}\right)
=\frac{\partial^{|u|}\M}{\partial \bo{\x}_{u}}\left(\bo{\x}\right) $) then we have 
\begin{eqnarray}   \label{eq:decvtief1f}                                         
\var(\tefg)  & = &             
\int_{\Omega\times \Omega_u} 
\frac{\partial^{|u|} \tefg}{\partial \bo{\x}_{u}}\left(\bo{\x}\right) 
\frac{\partial^{|u|} \tefg^\T}{\partial \bo{\x}_{u}}\left(\bo{\x}'_{u} ,\bo{\x}_{\sim u}\right) 
\prod_{k \in u} \frac{F_{k}\left[\min(\x_{k},\,\x_{k}')\right] -F_{k}(\x_{k})F_{k}(\x_{k}')}{\rho_{k} (\x_{k})\rho_{k} (\x_{k}')}  \, \nonumber \\ 
 &  &    \times        
d\mu(\bo{\x}) d\mu(\bo{\x}_u')          
\,  .     
\end{eqnarray}   

$\quad$ (iii) If assumption (A6) and (A7) ($\frac{\partial^{|u|}\tefg}{\partial \bo{\x}_{u}}\left(\bo{\x}_u', \bo{\x}_{\sim u}\right)\rho_u\left( \bo{x}_u\right) = \frac{\partial^{|u|}\tefg}{\partial \bo{\x}_{u}}\left(\bo{\x}\right)\rho_u\left( \bo{x}_u'\right)$) hold then                       	 
\begin{eqnarray}   \label{eq:decvtief2f}                                         
\var(\tefg)  & =  & \frac{1}{2^{|u|}}                            
\int_{\Omega} 
\frac{\partial^{|u|} \tefg}{\partial \bo{\x}_{u}}\left(\bo{\x}\right) 
\frac{\partial^{|u|} \tefg^\T}{\partial \bo{\x}_{u}}\left(\bo{\x}\right) 
\prod_{k \in u} \frac{F_{k}(\x_{k})\left(1 -F_{k}(\x_{k})\right)}{\left[\rho_{k} (\x_{k})\right]^2}  \, d\mu(\bo{\x}) \,  , 
\end{eqnarray}            
\end{corollary}               
\begin{proof}
See Appendix \ref{app:equa}.               
\end{proof}      
           
Remark that the function $\tefg_o$ defined in (\ref{eq:ex1}) satisfies assumptions (A6)-(A7), and therefore the integral equalities in (\ref{eq:decvtief1f}-\ref{eq:decvtief2f}) make sense. Equation (\ref{eq:decvtief2f}) makes use of the evaluations of the derivatives for only one sample of input values to compute the variance of $\tefg$. This property is practically interesting for functions or models ($\M$ or $\tefg$) that take a lot of time for one model run. For functions that do not satisfy assumption (A7), Equation (\ref{eq:decvtief1f}) can be used to compute the variance despite it requires the evaluations of the derivatives for two samples of input values.  For such functions, it is interesting to have upper bounds of their variances using only one sample evaluations of the derivatives, as a small value of an upper bound means  that the considered function is nearly a constant. Thus, Theorem \ref{theo:ineq} gives an integral inequality that involves both quantities. 
                  
\begin{theorem}        \label{theo:ineq}
Let $\tefg \in \mathcal{T}^{u, d}$ denote a function, $\bo{X}$ be $d$ input factors, $F_j$ (resp. $\rho_j$) be the CDF (resp. the PDF) of $X_j$ with $j\in \{1,\, \ldots,\, d\}$, and $\bo{x}$ be a sample value of $\bo{X}$. If assumptions (A1)-(A6)  hold then \\

$\quad$ (i) we have the following inequality:              
\begin{eqnarray}   \label{eq:Intief}                                          
\var(\tefg)  & \matleq  & \frac{1}{2^{|u|}}                            
\int_{\Omega} 
\frac{\partial^{|u|} \tefg}{\partial \bo{\x}_{u}}\left(\bo{\x}\right) 
\frac{\partial^{|u|} \tefg^\T}{\partial \bo{\x}_{u}}\left(\bo{\x}\right) 
\prod_{k \in u} \frac{F_{k}(\x_{k})\left(1 -F_{k}(\x_{k})\right)}{\left[\rho_{k} (\x_{k})\right]^2}  \, d\mu(\bo{\x}) \,  ,  
\end{eqnarray}       
where $\matleq$ is the symbol for the Loewner partial order on matrices, that is, $\Sigma_1 \matleq \Sigma_2$ if $\Sigma_2 -\Sigma_1$ is a positive semi-definite matrix. \\            
             
$\quad (ii) $ The equality holds for the function defined in (\ref{eq:ex1}), that is, 
\begin{equation}
\tefg_0(\bo{x}) = \prod_{j \in \bar{u}} \left[\bo{a}_j F_j(x_j) + \bo{b}_j \right] \prod_{j \in u} \left[\bo{a}_j F_j(x_j) - \frac{\bo{a}_j}{2} \right] \, .      
\end{equation}      
 where $\bo{a}_j, \, \bo{b}_j  \in \R^\NN$.                
\end{theorem}    
\begin{proof}        
See Appendix  \ref{app:ineq}.                          
\end{proof}   
 
Theorem \ref{theo:ineq} establishes a new and optimal integral inequality for a specific class of functions, and this inequality is going to be used in  multivariate sensitivity analysis (see Section \ref{sec:msa}). However, it is a bit far from the Poincar\'e-type  inequalities, which involve mainly the variance of a function and the integral of the square of its gradient. In this sense, Corollary \ref{coro:wpt} provides an optimal weighted Poincar\'e's  inequality and the Poincar\'e  inequality for any smooth function $\M : \R \to \R^\NN$ and for a class of Borel probability measures.          
\begin{corollary} \label{coro:wpt}              
Let $\M : \R \to \R^\NN$ denote a function, $\bo{X}$ be $d$ input factors, $F_j$ (resp. $\rho_j$) be the CDF (resp. the PDF) of $X_j$ with $j\in \{1,\, \ldots,\, d\}$, and $\bo{x}$ be a sample value of $\bo{X}$. If $\M$ satisfies assumptions (A1)-(A5) then \\ 
       
$\quad$ (i) the optimal weighted-Poincar\'e inequality  is given by
\begin{equation}     \label{eq:wpm1}   
\var\left(f\right) \matleq \frac{1}{2}\int_{\Omega}  
\nabla \M \left(\x\right) \nabla \M^\T \left(\x\right)
 \frac{ F(\x) \left(1- F(\x)\right)}{\left[\rho(\x)\right]^2}
 \,  d\mu(\x) \, ,         
\end{equation}
where $\Sigma_1 \matleq \Sigma_2$ if $\Sigma_2 -\Sigma_1$ is a positive semi-definite matrix.\\
     
$\quad$ (ii) A new Poincar\'e inequality on the real line is given by
\begin{equation}     \label{eq:pm1}  
\var\left(f\right)  \matleq   \min \left(\frac{1}{2}
\sup_{\x \in \Omega} \frac{ F(\x) \left(1- F(\x)\right)}{\left[\rho(\x)\right]^2}, \, 
4 \left[\sup_{\x \in \Omega} \frac{ F(\x) \left(1- F(\x)\right)}{\rho(\x)}\right]^2
\right)
\int_{\Omega}  
\nabla \M \left(\x\right) \nabla \M^\T \left(\x\right)
 \,  d\mu(\x) \, .     
\end{equation} 
\end{corollary}             
\begin{proof}       
Point (i) is obvious bearing in mind Theorem \ref{theo:ineq} and Remark \ref{rem:var}. Indeed, Remark \ref{rem:var} shows that $\M$ satisfies assumption (A6). \\
Bearing in mind the supremum properties, Point (ii) becomes obvious  using Point (i) and Corollary 1 from \cite{roustant17}. 
\end{proof}               
         
\begin{rem}             
Corollary \ref{coro:wpt} gives an interesting inequality that generalizes the  inequality obtained in \cite{hardy73} and used in \cite{kucherenko16}. 
Indeed, for a function $f : \R \to \R$ ($\NN=1$), our inequality (this paper) is reduced to             
\begin{equation}     \label{eq:wp1}  
\var\left(\M\right) \leq \frac{1}{2}\int_{\Omega}  
\left[\nabla \M \left(\x\right)\right]^2  
 \frac{ F(\x) \left(1- F(\x)\right)}{\rho(\x)^2}
 \,  d\mu(\x) \, .  
\end{equation}
In the  case of the standard uniform distribution ($F(x)=x$ and $\rho(\x)=1$), Equation (\ref{eq:wp1}) was already established in \cite{hardy73}, and it has been shown in  \cite{kucherenko16} that the new integral inequality allows for improving the optimal Poincar\'e inequality associated with the standard uniform distribution.
\end{rem}     
   
Based on Corollary \ref{coro:wpt}, we are going to give a general result about integral inequalities for any $d$-dimension function $\M : \R^d \to \R^\NN$. Of course, a general result can be inefficient for a specific function. For instance, an optimal inequality for an additive function (i.e., $\M(\bo{x}) =\sum_{j=1}^d \M_j(x_j)$) can be inefficient for a function of the form $\M(\bo{x}) =\prod_{j=1}^d \M_j(x_j)$ (multiplicative function). Thus, including some information about the form of a function can help improving some integral inequalities.  

            
\begin{defi}  
Consider a function $\M : \R^d \to \R^\NN$ that includes $d$ variables and $\frac{\partial^|u| \M}{\partial \bo{\x}_u}$ its cross-partial derivative w.r.t. $\bo{x}_u$ with $u\subseteq \{1,\, \ldots,\, d\}$.  For all $j \in \{1,\, \ldots,\, d\}$, the theoretical interaction super-set of $\X_j$ is the set $\mathcal{A}_j$ given by
\begin{equation}    
\mathcal{A}_j := \left\{\{j,\, u\} :  \text{such that}\, \frac{\partial^{|u|+1} \M}{\partial \bo{\x}_{\{j,\, u\}}} \neq \bo{0}, \,  u \subseteq \{1,\, \ldots,\,  d\}\backslash \{j\}  \right\} \, .    
\end{equation} 
\end{defi} 
         
The interaction set $\mathcal{A}_j$ gives the information about the variables that  interact with $X_j$ and the order of interactions regardless the importance of these interactions in practice. The cardinal of $\mathcal{A}_j$ (i.e., $|\mathcal{A}_j|$) can be seen as the number of the active components in ANOVA decomposition in a direction driven by $\X_j$. Of course, we have $1\leq |\mathcal{A}_j| \leq 2^{d-1}$, $|\mathcal{A}_j| =1$ 
for an additive function  and $|\mathcal{A}_j| =2^{d-1}$ for a multiplicative function. Using the definition of the set $\mathcal{A}_j$, Theorem \ref{theo:pweqineq} generalizes results from Corollary \ref{coro:wpt} to cope with  a $d$-dimension function $\M : \R^d \to \R^\NN$.                    
            
\begin{theorem} \label{theo:pweqineq}  
Let $\M : \R^d \to \R^\NN$ denote a function, $\bo{X}$ be $d$ input factors, $F_j$ (resp. $\rho_j$) be the CDF (resp. the PDF) of $X_j$ with $j\in \{1,\, \ldots,\, d\}$,  $\bo{x}$ be a sample value of $\bo{X}$ and assume that assumptions (A1)-(A5) hold. \\ 

$\quad$ (i) A general integral inequality is given by 
\begin{equation}   
 \var(\M)   \matleq \sum_{\substack{u, \, u \subseteq \{1,\, \ldots,\, d\}\\ |u|>0}} 
		\frac{1}{2^{|u|}} \int_{\Omega} 
\frac{\partial \M}{\partial \bo{\x}_u}\left(\bo{\x}\right) 
\frac{\partial \M^\T}{\partial \bo{\x}_u}\left(\bo{\x}\right) 
 \prod_{k \in u} \frac{F_{k}(\x_{k})\left(1 -F_{k}(\x_{k})\right)}{\left[\rho_{k} (\x_{k})\right]^2}  \, d\mu(\bo{\x})  \,  ,      
\end{equation} 
where $\Sigma_1 \matleq \Sigma_2$ if $\Sigma_2 -\Sigma_1$ is a positive semi-definite matrix.\\
   
$\quad$ (ii) If there exists $J_0 \subseteq \{1,\, \ldots,\, d\}$ such that $\int_{\Omega_j} \M \, d\mu_j =\bo{0},\, \forall\, j\in J_0$ then we have     
\begin{equation}
 \var(\M)   \matleq \frac{1}{2}                          
\int_{\Omega} 
\frac{\partial \M}{\partial \x_{(1)}}\left(\bo{\x}\right) 
\frac{\partial \M^\T}{\partial \x_{(1)}}\left(\bo{\x}\right) 
 \frac{F_{(1)}(\x_{(1)})\left(1 -F_{(1)}(\x_{(1)})\right)}{\left[\rho_{(1)} (\x_{(1)})\right]^2}  \, d\mu(\bo{\x}) \,  ,         
\end{equation}       
with the integer  $(1)$ given by 
$$(1) = \argmin{j \in J_0} \int_{\Omega}             
\frac{\partial \M}{\partial \x_{j}}\left(\bo{\x}\right) 
\frac{\partial \M^\T}{\partial \x_{j}}\left(\bo{\x}\right) 
 \frac{F_{j}(\x_{j})\left(1 -F_{j}(\x_{j})\right)}{\left[\rho_{j} (\x_{j})\right]^2} \, d\mu(\bo{\x}) \, .$$        
      
$\quad$ (iii) If the set $\mathcal{A}_j$ is known for all $j \in \{1,\,\ldots,\,d\}$ then
\begin{equation}          
 \var(\M)   \matleq \sum_{k=1}^d \frac{|\mathcal{A}_{(j)}\backslash \cup_{i \in \{(1),\, \ldots,\, (j-1)\} } \mathcal{A}_{(i)} |}{2}                        
\int_{\Omega}      
\frac{\partial \M}{\partial \x_{(j)}}\left(\bo{\x}\right) 
\frac{\partial \M^\T}{\partial \x_{(j)}}\left(\bo{\x}\right) 
 \frac{F_{(j)}(\x_{(j)})\left(1 -F_{(j)}(\x_{(j)})\right)}{\left[\rho_{(j)} (\x_{(j)})\right]^2}  \, d\mu(\bo{\x}) \,  ,       
\end{equation}                   
where the integer $(j)$ is given by   
$$(j) = \argmin{j \in \{1,\,\ldots, \, d\}\backslash \{(1),\, \ldots,\, (j-1)\}} \int_{\Omega}      
\frac{\partial \M}{\partial \x_{j}}\left(\bo{\x}\right) 
\frac{\partial \M^\T}{\partial \x_{j}}\left(\bo{\x}\right) 
 \frac{F_{j}(\x_{j})\left(1 -F_{j}(\x_{j})\right)}{\left[\rho_{j} (\x_{j})\right]^2}  \, d\mu(\bo{\x}) \, ,$$     
and $|\mathcal{A}_{(j)} |$ is the cardinal of the set $\mathcal{A}_{(j)}$.  
\end{theorem} 
\begin{proof} 
See Appendix \ref{app:pweqineq}. 
\end{proof}

\section{Application to multivariate sensitivity analysis} \label{sec:msa}
Multivariate sensitivity analysis, including variance-based sensitivity, allows for identifying the most important input factors of a function or model $\M : \R^d \to \R^\NN$. It is based on what we call sensitivity functionals, which aim to measure the single and overall contributions of input factors over the whole model outputs. While, the first-order functional, that is, 
$$ 
\M^{\text{fo}}_u :=    
\esp \left[\M(\bo{\x})| \bo{\X}_{u}\right] -\esp \left[\M(\bo{\X})\right] \, , 
$$ 
with $\esp \left[\M(\bo{\X})\right]:=\int_\Omega \M \, d\mu$,  
is used to assess the single contribution of  $\bo{\X}_u$, we use the total-effect functional (TEF), given by (\cite{lamboni18,lamboni18a}) 
 $$ 
\tef_u(\bo{\X}) := \M(\bo{\X}) - \esp \left[\M(\bo{\X})| \bo{\X}_{\sim u}\right] \, , 
$$     
 to measure the overall contribution of $\bo{\X}_u$, including interactions. The functional $\tef_u(\bo{\X})$ captures all the single contribution of input in $\bo{\X}_u$ and all the interactions between each input factor in $\bo{\X}_u$ and the other inputs.      
Likewise, the total-interaction effect functional (TIEF) allows for assessing the single contribution of $\bo{\X}_u$ and the interaction between the latter and the other inputs ($\bo{\X}_{\sim u}$). It is given as follows (\cite{lamboni19}):  
\begin{eqnarray}  \label{eq:tieftef0} 
 \tief_{u}(\bo{\X}) &:=& \sum_{\substack{v,\, v \subset u}}  
 (-1)^{|u|-|v| +1} \tef_{u \backslash v}(\bo{\X})   \, .         
\end{eqnarray}             
  
Such sensitivity functionals (SFs) are random vectors, and their components may be correlated. Proposition \ref{prop:profun} gives their properties.       
     
\begin{prop} \label{prop:profun} If assumption (A1) holds, we have \\

$\quad$ (i)  $\M(\bo{\X})$ does not depend on $\X_j$ if and only if  the first-order and total SFs are null, that is,        
\begin{equation}
\M^{\text{fo}}_j(\X_j) = \M^{tot}_j(\bo{\X}) = \bo{0} \, . 
\end{equation}    
                             
$\quad$ (ii)  $\M(\bo{\X})$ depends only on $\X_j$ if and only if 
\begin{equation}
 \M^{\text{fo}}_j(\X_j) =\M^{tot}_j(\bo{\X}) =\M(\X_j) - \esp\left[\M(\X_j) \right] \, .      
\end{equation}            
\end{prop}       
\begin{proof}  
The proof is obvious bearing in mind the properties of the conditional expectation.    
\end{proof} 
  
Using the SFs, we can define many types of generalized sensitivity indices by properly choosing the importance measure. One may use the moment dependent measure such as the variance or the moment independent measure such as the probability density, distribution or a Hilbert-Schmidt information criterion. In what follows, we use the variance as a measure of importance. When we use the variance as a measure of the variability of the model outputs, a definition of the sensitivity indices for the multivariate response models should be based on the variances of the SFs. \\          
 
The first-order covariance matrix  of $\bo{\X}_u$ is $\D_u$ given by         
\begin{equation}    \label{eq:siguf}  
\D_u := \var\left[\esp\left(\esp\M(\bo{\X})\, |\, \bo{\X}_u \right)\right] \, .        
\end{equation}                                          
  
Further, the total-effect covariance matrix of $\bo{\X}_u$ is $\D_u^{tot}$ given by     
\begin{equation}    \label{eq:sigut} 
\D_u^{tot} := \var\left[\M(\bo{\X}) - \esp \left[\M(\bo{\X})| \bo{\X}_{\sim u}\right]\right] \, .                 
\end{equation}                 
      
Likewise, the total-interaction covariance matrix of $\bo{\X}_u$ is $\D_u^{sup}$ given by
\begin{equation}  \label{eq:siguti}
\D_u^{sup} := \var\left[\tief_{u}(\bo{\X})\right] \, .   
\end{equation}    
                  
For the single-response models ($\NN=1$), the prioritization of input factors based on the covariance matrices is straightforward, as the covariance matrices are scalars. In the case of the multivariate-response models ($\NN >1$), Lamboni \cite{lamboni18a} proposed to apply matrix norms on the covariance matrices such as $\D_u, \, \D_u^{sup}, \, \D_u^{tot}$ in order to  prioritize input factors. In this paper, we consider two types of generalized sensitivity indices (GSIs) from \cite{lamboni18}. The first-type GSIs are obtained by applying the trace on the covariance matrices. This first-type GSIs are well-suited for non-correlated components of the SFs. When the components of the SFs are correlated, the second-type GSIs allow for capturing such correlations. Both types of GSIs are given in Definition \ref{def:gsi}. 
          
\begin{defi}               \label{def:gsi} 
Consider $\Sigma$, $\D_u$, $\D_u^{sup}$, and $\D_u^{tot}$ the covariance matrices of the model outputs, the first-order, the total-interaction effect and the total-effect functionals,   respectively. \\            

The classical GSIs (\cite{lamboni09,lamboni11,gamboa14,lamboni18a}) are defined below. \\          
The first-order GSI of $\bo{X}_u$ is defined as follows:  
\begin{equation}        \label{eq:gsif}   
GSI_u^F := \frac{\normf{\D_u^{1/2}}^2}{\normf{\Sigma^{1/2}}^2} \, .
\end{equation}     
Further, the total GSI of $\bo{X}_u$ is given by   
\begin{equation}  \label{eq:gsitf} 
GSI_{T_u}^F := \frac{\normf{\left(\D_u^{tot}\right)^{1/2}}^2}{\normf{\Sigma^{1/2}}^2} \, , 
\end{equation}    
and the total-interaction GSI of $\bo{X}_u$ is given by    
\begin{equation}  \label{eq:gsitif} 
GSI_{sup, u}^F := \frac{\normf{\left(\D_u^{sup}\right)^{1/2}}^2}{\normf{\Sigma^{1/2}}^2} \, . 
\end{equation}           
         
Likewise, the second-type GSIs (\cite{lamboni18a}) are defined as follows:
\begin{equation}        \label{eq:gsil2}      
GSI_u^{l_2} := \frac{\normf{\D_u}}{\NN\normf{\Sigma^{1/2}}^2} \, ; 
\end{equation}     
\begin{equation}  \label{eq:gsitl2} 
GSI_{T_u}^{l_2} := \frac{\normf{\D_u^{tot}}}{\NN \normf{\Sigma^{1/2}}^2} \, ; 
\end{equation}   
and    
\begin{equation}  \label{eq:gsitil2} 
GSI_{sup, u}^{l_2} := \frac{\normf{\D_u^{sup}}}{\NN \normf{\Sigma^{1/2}}^2} \, . 
\end{equation}    
\end{defi}                                   
    
\begin{rem}   
The first-type GSIs such as $GSI_u^F,\, GSI_{T_u}^F, \, GSI_{sup, u}^F$ are equivalent to the classical definition, that is, 
$$
GSI_u^F =\frac{\trace\left(\D_u \right)}{\trace\left(\Sigma\right)},\quad 
GSI_{T_u}^F= \frac{\trace\left(\D_u^{tot} \right)}{\trace\left(\Sigma\right)},\quad 
GSI_{sup, u}^F = \frac{\trace\left(\D_u^{sup} \right)}{\trace\left(\Sigma\right)}\, . 
$$
 
In the case of single response models ($\NN=1$), both types of GSIs come down to Sobol' indices.                     
\end{rem}     
           
\subsection{Derivative-based GSIs and proxy-measure for multivariate sensitivity analysis} \label{sec:gside}
The sample-based computations of GSIs in Definition \ref{def:gsi} can required a lot of model runs, which can become intractable in the case of complex models that require a lot of time for one model run. The upper bounds of the total and total-interaction GSIs can be used for quickly selecting the most important input factors (\cite{kucherenko09,lamboni13}). This section provides the derivative-based expressions of sensitivity functionals  and  new upper bounds of the GSIs. \\             
 
The derivative-based expressions of the TIEF in (\ref{eq:tieftef0}) is given by (see Lemma \ref{lem:ex})     
\begin{eqnarray}  \label{eq:tiefdef1} 
 \tief_{u}(\bo{\X}) &=& (-1)^{|u|+1} \int_{\Omega_u}
\frac{\partial^{|u|} \M}{\partial \bo{\x}_{u}}\left(\bo{\x}'_{u},\bo{\X}_{\sim u}\right) 
 \prod_{j \in u} \frac{F_{j}(\x_{j}') - \indic_{[\x_{j}' \geq \X_{j}]}}{\rho_{j} (\x_{j}') }  \, d\mu(\bo{\x}'_{u}) 
  \, .        
\end{eqnarray}    
For a single input $\X_j$, the TIEF comes down to the total-effect functional (TEF).  \\ 

Likewise, it comes from Theorem \ref{theo:deriaov1} that the derivative-based first-order functional is given by  
\begin{eqnarray}  \label{eq:fodef1} 
\M^{\text{fo}}_j(\x_j) &=& \int_{\Omega} \frac{\partial \M}{\partial \x_{j}}\left(\bo{\x}'\right) 
 \frac{F_{j}(\x_{j}') - \indic_{[\x_{j}' \geq \x_{j}]}}{\rho_{j} (\x_{j}') }  \, d\mu(\bo{\x}')  \, .         
\end{eqnarray}

\subsubsection{Upper bound of the first-order generalized sensitivity indices}
By Corollary \ref{coro:equa}, the first-order covariance matrix of $\X_j$ (i.e., $\Sigma_j=\var \left(\M^{\text{fo}}_j \right)$) is given by
         
\begin{eqnarray}     
\Sigma_j & = & \int_{\Omega^2} 
\frac{\partial \M}{\partial \x_{j}}\left(\bo{\x} \right) 
\frac{\partial\M^\T}{\partial \x_j}\left(\bo{\x}'\right)   
\times  \frac{F_{j}\left[\min(\x_{j},\,\x_{j}')\right] -F_{k}(\x_{j})F_{k}(\x_{j}')}{\rho_{j} (\x_{j})\rho_{j} (\x_{j}')}    
d\mu(\bo{\x}) d\mu\bo{\x}')     
  \, ,    
\end{eqnarray}
and by Theorem \ref{theo:ineq}, the upper bound of $\Sigma_j$ comes from the inequality  
     
\begin{eqnarray}     
\Sigma_j  &\matleq  & \frac{1}{2}               
\int_{\Omega} 
\frac{\partial \M}{\partial \x_j}\left(\bo{\x}\right) 
\frac{\partial \M^\T}{\partial \x_{j}}\left(\bo{\x}\right) 
\frac{F_{j}(\x_{j})\left(1 -F_{j}(\x_{j})\right)}{\left[\rho_{j} (\x_{j})\right]^2}  \, d\mu(\bo{\x}) \, . 
\end{eqnarray} 
        
The computation of the upper bound of $\Sigma_j$, that is, 
$$ 
N\!U\!B_j:=\frac{1}{2}                
\int_{\Omega} \frac{\partial \M}{\partial \x_j}\left(\bo{\x}\right) 
\frac{\partial \M^\T}{\partial \x_{j}}\left(\bo{\x}\right) 
\frac{F_{j}(\x_{j})\left(1 -F_{j}(\x_{j})\right)}{\left[\rho_{j} (\x_{j})\right]^2}  \, d\mu(\bo{\x}) \, , 
$$ 
will require the evaluation of the model derivatives for only one sample of input values.   
            
\subsubsection{Upper bound of the total and total-interaction generalized sensitivity indices}
By Corollary \ref{coro:equa}, the total-interaction covariance matrix  (i.e., $\Sigma_u^{sup}=\var\left(\tief_u\right)$) is given by       
\begin{eqnarray}   \label{eq:decvtief1}                                         
\Sigma_u^{sup}  & = &               
\int_{\Omega\times \Omega_u} 
\frac{\partial^{|u|} \M}{\partial \bo{\x}_{u}}\left(\bo{\x}\right) 
\frac{\partial^{|u|} \M^\T}{\partial \bo{\x}_{u}}\left(\bo{\x}'_{u} ,\bo{\x}_{\sim u}\right) 
\prod_{k \in u} \frac{F_{k}\left[\min(\x_{k},\,\x_{k}')\right] -F_{k}(\x_{k})F_{k}(\x_{k}')}{\rho_{k} (\x_{k})\rho_{k} (\x_{k}')}  \, \nonumber \\    
 &  &    \times        
d\mu(\bo{\x}) d\mu(\bo{\x}')          
\,  ,
\end{eqnarray}                   
			
and by Theorem \ref{theo:ineq}, the upper bound of $\Sigma_u^{sup}$ comes from the inequality
       
\begin{eqnarray}   \label{eq:ubtief}                                          
\Sigma_u^{sup}  & \matleq  & \frac{1}{2^{|u|}}             
\int_{\Omega}   
\frac{\partial^{|u|} \M}{\partial \bo{\x}_{u}}\left(\bo{\x}\right) 
\frac{\partial^{|u|} \M^\T}{\partial \bo{\x}_{u}}\left(\bo{\x}\right) 
\prod_{k \in u} \frac{F_{k}(\x_{k})\left(1 -F_{k}(\x_{k})\right)}{\left[\rho_{k} (\x_{k})\right]^2}  \, d\mu(\bo{\x}) \,  .   
\end{eqnarray}          
       
For a single input $\X_j$ ($u=\{j\}$), the TIEF functional is equal to the TEF, and we deduce the upper bound of $\Sigma_j^{tot}$ from the inequality

\begin{eqnarray}   \label{eq:ubtef}                                           
\Sigma_j^{tot}  & \matleq  & \frac{1}{2}             
\int_{\Omega} 
\frac{\partial \M}{\partial \x_{j}}\left(\bo{\x}\right) 
\frac{\partial{j} \M^\T}{\partial \x_{j}}\left(\bo{\x}\right) 
 \frac{F_{j}(\x_{j})\left(1 -F_{j}(\x_{j})\right)}{\left[\rho_{j} (\x_{j})\right]^2}  \, d\mu(\bo{\x}) \,  .  
\end{eqnarray}     
 
We can see that the upper bound of $\Sigma_j$ (i.e., $N\!U\!B_j$) is equal to the upper bound of $\Sigma_j^{tot}$. This result is not surprising because the first-order GSI is equal to the total GSI for additive models such as  $\M(X_1,X_2)=  X_1 + X_2$.                  
          
\subsection{Estimator of the proxy-measure of generalized sensitivity indices}\label{sec:proxy} 
This section deals with the estimators of the upper bounds of the GSIs, which will be used for computing the proxy-measure of the GSIs. \\
 
The upper bound of the non-normalized total-interaction GSI is  given by 
\begin{equation} \label{eq:UBdef} 
N\!U\!B_u^{sup} := \frac{1}{2^{|u|}}             
\int_{\Omega} 
\frac{\partial^{|u|} \M}{\partial \bo{\x}_{u}}\left(\bo{\x}\right) 
\frac{\partial^{|u|} \M^\T}{\partial \bo{\x}_{u}}\left(\bo{\x}\right) 
\prod_{k \in u} \frac{F_{k}(\x_{k})\left(1 -F_{k}(\x_{k})\right)}{\left[\rho_{k} (\x_{k})\right]^2}  \, d\mu(\bo{\x}) \, ,
\end{equation}
and we can compute $N\!U\!B_u^{sup}$ using the following estimator.
 
\begin{theorem}        \label{theo:est}
Let $\bo{\X}_i$, $i=1,\, 2,\, \ldots, \, m$  be a sample of size $m$ from $\bo{\X}$. If assumptions (A1)-(A5)  hold then \\  
        
 $\quad (i) $ the unbiased estimator of $N\!U\!B_u^{sup}$ is given by  
\begin{eqnarray} \label{eq:ubest}     
\widehat{N\!U\!B_u^{sup}} & := &  \frac{1}{2^{|u|} m}\sum_{i=1}^m     
\frac{\partial^{|u|} \M}{\partial \bo{\x}_{u}}\left(\bo{\X}_i\right) 
\frac{\partial^{|u|} \M^\T}{\partial \bo{\x}_{u}}\left(\bo{\X}_i\right) 
\prod_{k \in u} \frac{F_{k}(\X_{i,k})\left(1 -F_{k}(\X_{i,k})\right)}{\left[\rho_{k} (\X_{i,k})\right]^2} \, ,  
\end{eqnarray}                             
 with                
\begin{equation}   
\esp\left(\widehat{N\!U\!B_u^{sup}}\right)  = N\!U\!B_u^{sup}  \, , 
\end{equation}
and we have 
\begin{equation}   
\widehat{N\!U\!B_u^{sup}} \xrightarrow{\mathcal{P}} N\!U\!B_u^{sup} \, ,
\end{equation}         
when  $m \to +\infty$, with $\xrightarrow{P}$ the convergence in probability. \\  
     
$\quad (ii) $ The unbiased estimator of $\trace\left(N\!U\!B_u^{sup}\right)$ is given by                                                  
\begin{equation} \label{eq:ubestF}     
\widehat{\trace\left(N\!U\!B_u^{sup}\right)}  :=   \frac{1}{2^{|u|} m}\sum_{i=1}^m \trace\left[    
\frac{\partial^{|u|} \M}{\partial \bo{\x}_{u}}\left(\bo{\X}_i\right) 
\frac{\partial^{|u|} \M^\T}{\partial \bo{\x}_{u}}\left(\bo{\X}_i\right) 
\right]
\prod_{k \in u} \frac{F_{k}(\X_{i,k})\left(1 -F_{k}(\X_{i,k})\right)}{\left[\rho_{k} (\X_{i,k})\right]^2} \, ,       
\end{equation}                                 
and  we have \\                     
\begin{equation}      
\widehat{\trace\left(N\!U\!B_u^{sup}\right)} \xrightarrow{\mathcal{P}}  \trace\left(N\!U\!B_u^{sup}\right)
\quad  \text{when} \; m \to +\infty\, .                        
\end{equation}
     
$\quad (iii) $ The consistent estimator of $\normf{N\!U\!B_u^{sup}}$ is given by  
\begin{equation} \label{eq:ubestl2}     
\widehat{\normf{N\!U\!B_u^{sup}}}  :=  \frac{1}{2^{|u|} m} \normf{ 
\sum_{i=1}^m    
\frac{\partial^{|u|} \M}{\partial \bo{\x}_{u}}\left(\bo{\X}_i\right) 
\frac{\partial^{|u|} \M^\T}{\partial \bo{\x}_{u}}\left(\bo{\X}_i\right) 
\prod_{k \in u} \frac{F_{k}(\X_{i,k})\left(1 -F_{k}(\X_{i,k})\right)}{\left[\rho_{k} (\X_{i,k})\right]^2}
} \, ,       
\end{equation}                                 
and we have   
\begin{equation}    
\widehat{\normf{N\!U\!B_u^{sup}}} \xrightarrow{\mathcal{P}}  \normf{N\!U\!B_u^{sup}} \, .
\end{equation}     
when $m \to +\infty$, with $\xrightarrow{P}$ the convergence in probability. 
\end{theorem}                                                
\begin{proof} 
Point (i) is obvious using the definition of $N\!U\!B_u^{sup}$ in (\ref{eq:UBdef}), the linearity of the expectation,  and the law of large numbers. \\
Point (ii) is a consequence of Point (i) knowing the definition of the trace ($\trace$) and using the linearity of $\trace$. \\      
Point (iii) is obtained by combining Point (i), the definition of $\normf{\cdot}$, and the Slutsky theorem.          
\end{proof}            

Now, we are going to derive the estimators of the upper bounds of both types of the normalized GSIs. Theorem \ref{theo:UBngsi} provides the proxy-measures of the total-interaction GSIs.     
      
\begin{theorem}        \label{theo:UBngsi}  
Under assumptions (A1)-(A5), we have the following proxy-measure of GSIs.\\    

 $\quad (i)$ The first-type proxy-measure of the total-interaction GSI of $\bo{\X}_u$ is given by  
\begin{equation}   \label{eq:gsiestub} 
 \widehat{U\!B_u^F} := \frac{\widehat{\trace\left(N\!U\!B_u^{sup} \right)}}{\trace\left(\widehat{\Sigma} \right)} \, ,   
\end{equation}      
where $\widehat{\Sigma}$ is the classical estimator of the variance of the model outputs, that is,  $\Sigma=\var(\M)$. \\                      
                 
If  $m \to +\infty$  we have        
\begin{equation} 
   \widehat{U\!B_u^F} \xrightarrow{\mathcal{P}} \frac{\trace\left(N\!U\!B_u^{sup}\right)}{\trace\left(\Sigma\right)}\, ,
\end{equation}      
with $\xrightarrow{P}$ the convergence in probability. \\ 
    
$\quad (ii)$ The second-type proxy-measure of the total-interaction GSI of $\bo{\X}_u$ is given by           
\begin{equation}  
 \widehat{U\!B_{u}^{l_2}} := \frac{\widehat{\normf{N\!U\!B_u^{sup}}}}{\NN \trace\left(\widehat{\Sigma} \right)} \, ,         
\end{equation}   
with  
\begin{equation}         
   \widehat{U\!B_{u}^{l_2}} \xrightarrow{\mathcal{P}} \frac{\normf{N\!U\!B_u^{sup}}}{\NN \trace\left(\Sigma\right)}\, , 
\end{equation}    
when $m \to \infty$, with $\xrightarrow{P}$ the convergence in probability.    
\end{theorem}       
\begin{proof}  
Point (i) and Point (ii) are obvious bearing in mind Theorem \ref{theo:est} and the Slutsky theorem.               
\end{proof}              
                                                
\subsection{Analytical test case ( Any $d$ and $\NN=1$)} 
We consider the class of functions defined as follows:
\begin{equation} \label{eq:fcdf}   
f(\bo{\x}) = \prod_{j=1}^d g_j(x_j),\; \quad \text{with} \quad  \; g_j(x_j) = a_j F_j(x_j)+b_j\, ,                           
\end{equation}  
where $F_j$ is the CDF of $\X_j$ ($\X_j \sim F_j$), $a_j,\, b_j \in \R$ with $j=1,\,2,\, \ldots, \, d$. \\    

The function $f$ includes $d$ independent input factors following the CDF $F_j,\, j=1,\,2,\, \ldots, \, d$. Knowing that $\X_j \sim F_j   \Longleftrightarrow   F_j(\X_j) \sim \mathcal{U}[0,\, 1]$, we can write    
$$
\var\left[g_j(X_j)\right]=a_j^2\var\left[F_j(\X_j) \right]=\frac{a_j^2}{12} \, ,   
$$
and   
$$
\esp\left[F_j(\X_j)\left(1-F_j(\X_j)\right) \right] = \int_0^1 x(1-x)\, dx=\frac{1}{6}
\, .   
$$  
           
The analytical value of the non-normalized total index is given by
$$
\Sigma_j^{tot} = \esp\left[ \var\left(\M(\bo{\X}) \, |\, \X_{\sim j}  \right)\right]
= \frac{a_j^2}{12} \prod_{i=1\, i\neq j}^d \esp\left[g_i(\X_i)^2 \right]
\, ,           
$$   
and it upper bound is given by (Equation (\ref{eq:ubtef}))                              
\begin{eqnarray}
N\!U\!B_j^{tot}  &=&\frac{1}{2} \esp\left[a_j^2 F_j(\X_j)\left(1-F_j(\X_j)\right) \prod_{i=1\, i\neq j}^d g_i(\X_i)^2  \right] \nonumber \\
 & =&  \frac{a_j^2}{2} \esp\left[F_j(\X_j)\left(1-F_j(\X_j)\right) \right] \prod_{i=1\, i\neq j}^d \esp\left[g_i(\X_i)^2 \right] \nonumber \\
 &=& \frac{a_j^2}{12} \prod_{i=1\, i\neq j}^d \esp\left[g_i(\X_i)^2 \right]
\nonumber \\
 &=& \Sigma_j^{tot} \, . \nonumber  
\end{eqnarray}       
          
It comes out that the proxy-measure is exactly the total index for this class of functions. Therefore, the proxy-measure will gives the right ranking of input factors.

\subsection{Numerical test cases}\label{sec:test}
In this section, we performed numerical tests to assess the effectiveness of our proxy-measure of GSIs. To illustrate our approach, we considered two types of functions as follows: functions with a small number of inputs ($d=3$), and functions with a medium number of inputs ($d=6$, $d=10$). We computed the model derivatives using the finite difference method, and  using Sobol's sequence or Quasi-Monte Carlo (\cite{dutang13}). For the sample size $m=1000$, we replicated the process of computing the proxy-measure of GSIs $R=30$ times by changing the seed randomly when sampling input values.   \\

In this paper, the functions considered belong to either a class of functions having many important input factors (all the inputs) by interactions  or a class of functions which are continuous but  differentiable almost everywhere (see last test case). The first class of functions often requires a lot of model runs to obtain reasonable estimates of sensitivity indices.        
         
\subsubsection{Ishigami's function ($d=3$, $\NN=1$)}
The Ishigami function includes three independent input factors following a uniform distribution on $[-\pi,\, \pi]$, and it  provides one output. It is defined as follows:          
\begin{equation} \label{eq:ishi0}        
\M(\bo{\x}) = \sin(\x_1)+ 7\sin^2(\x_2) + 0.1\, \x_3^4 \sin(\x_1)              
\, .              
\end{equation}     
For this function, all the input factors are important. The true values of Sobol' indices and the estimates of the proxy-measure are  listed in Table \ref{tab:ishi}.  
     
\begin{table}[ht]   
\centering  
\begin{tabular}{rrcccccc}
  \hline 
	  \hline     
		\multicolumn{1}{r}{} & \multicolumn{3}{c}{First-type proxy-measure}	&\multicolumn{1}{r}{} &	 \multicolumn{3}{c}{Second-type proxy-measure}\\       
	\cline{2-4}  	\cline{6-8}              
    & $GSI^F_{sup,u}$  & $U_u$ & $\widehat{U\!B_u^F}$ & & $GSI^{l2}_{sup,u}$ & $U_u^{l2}$ & $\widehat{U\!B_u^{l2}}$ \\     
 \cline{2-4}  	\cline{6-8}  
  $X_1$ & 0.558 & 2.230 & 1.682 (0.03) && 0.558 & 2.230 & 1.682 (0.03)  \\ 
  $X_2$ & 0.442 & 7.079 & 5.907 (0.07) &&  0.442 & 7.079 & 5.907 (0.07) \\ 
 $X_3$  & 0.244 & 3.174 & 0.867 (0.01) && 0.244 & 3.174 & 0.867 (0.01) \\  
 \hline                 
	\multicolumn{2}{r}{} & \multicolumn{2}{c}{} 	&\multicolumn{1}{r}{} &	\multicolumn{1}{r}{}	&	 \multicolumn{2}{c}{} \\         
	  \cline{2-4}  	\cline{6-8}                
$X_1:X_2$ &  0     & 0 & 9e-21 (4e-20) &&  0     & 0 & 9e-21 (4e-20)  \\ 
$X_1:X_3$ &  0.244 & 12.698 & 2.642 (0.03) &&  0.244 & 12.698 & 2.642 (0.03) \\        
$X_2:X_3$ &  0     & 0 & 4e-22 (e-21) && 0     & 0 & 4e-22 (e-21) \\    
   \hline                     
	  \hline                                                              
\end{tabular}           
\caption{True Sobol' indices and the estimates of the proxy-measure of these indices (average over $30$ replications, followed by their standard deviations in bracket) for the Ishigami function. We also added the best (known) Poincar\'e upper bound $U_u$ (\cite{lamboni13,roustant14}).   While, the top part of this table focuses on the total indices, the bottom part deals with the total-interaction indices. For concise reporting of small values, we use e-a for $10^{-a}$.}  
\label{tab:ishi}    
\end{table}

\subsubsection{Multivariate Ishigami's function ($d=3$, $\NN=3$)}
The multivariate Ishigami function includes three independent input factors following a uniform distribution on $[-\pi,\, \pi]$, and it  provides three outputs (\cite{lamboni18a}). It is defined as follows:          
\begin{equation} \label{eq:ishi}        
\M(\bo{\x}) =  \left[ \begin{array}{c}     
\sin(\x_1)+ 7\sin^2(\x_2) + 0.1\, \x_3^4 \sin(\x_1) \\  
\sin(\x_1)+ 5.896\sin^2(\x_2) + 0.1 \,\x_3^4 \sin(\x_1) \\
\sin(\x_1)+ 6.494\sin^2(\x_2) + 0.125 \,\x_3^4 \sin(\x_1)
\end{array}   
\right]             
\, .              
\end{equation}   
For this function, all the input factors are important regarding the whole outputs. 
The true values of the GSIs and the estimates of the proxy-measure are  listed in Table \ref{tab:ishim}.  
 
\begin{table}[ht]  
\centering  
\begin{tabular}{rrcccccc}
  \hline 
	  \hline     
		\multicolumn{1}{r}{} & \multicolumn{3}{c}{First-type proxy-measure}	&\multicolumn{1}{r}{} &	 \multicolumn{3}{c}{Second-type proxy-measure}\\       
	\cline{2-4}  	\cline{6-8}              
    & $GSI^F_{sup,u}$  & $U_u$ & $\widehat{U\!B_u^F}$ & & $GSI^{l2}_{sup,u}$ & $U_u^{l2}$ & $\widehat{U\!B_u^{l2}}$ \\     
 \cline{2-4}  	\cline{6-8}        
  $X_1$ & 0.628 & - & 1.921 (0.03) && 0.209 & - & 0.640 (0.01)  \\ 
  $X_2$ & 0.372 & - & 5.001 (0.06) &&0.125  & - & 1.666 (0.02)  \\ 
 $X_3$  & 0.284 & - & 1.014 (0.01) &&0.093  & - & 0.338 (0.004)  \\  
 \hline                    
	\multicolumn{2}{r}{} & \multicolumn{2}{c}{} 	&\multicolumn{1}{r}{} &	\multicolumn{1}{r}{}	&	 \multicolumn{2}{c}{} \\                    
	  \cline{2-4}  	\cline{6-8}           
$X_1:X_2$ &  0     & - & 7e-20 (3e-19)  &&  0  & - &  2.4e-20 (e-19)  \\ 
$X_1:X_3$ &  0.284 & - & 3.076 (0.03)   &&  0.093  & - & 1.025 (0.01)  \\        
$X_2:X_3$ &  0     & - & 4e-20 (2e-19)  &&0    & - & 1.2e-20 (6e-20) \\   
   \hline                  
	  \hline                                  
\end{tabular}                              
\caption{True generalized sensitivity indices and the estimates of the proxy-measure of these indices (average over $30$ replications, followed by their standard deviations in bracket) for the multivariate Ishigami function. While, the top part of this table focuses on the total indices, the bottom part deals with the total-interaction indices. For concise reporting of small values, we use e-a for $10^{-a}$.}          
\label{tab:ishim}
\end{table}                                                              
                
\subsubsection{Block-additive function ($d=6$, $\NN=1$)}   
The block-additive function includes six independent inputs following a uniform distribution on $[-1,\, 1]$, and it provides one output (\cite{roustant14}). It is defined as follows:
\begin{equation} \label{eq:pexpo}
    \M(\bo{\x}) = \cos(-0.8 -1.1 \x_1 + 1.1 \x_5 + \x_3)  + 
		\sin(0.5 +0.9 \x_4 + \x_2 -1.1 \x_6) 		\, .    
\end{equation} 
This function has been used to illustrate the improved Poincar\'e inequality in \cite{roustant14}. We used it for a comparison reason and for the fact that all the input factors are important. The true values of Sobol' indices and the estimates of the proxy-measure are  listed in Table \ref{tab:blockso}.             

\begin{table}[ht] 
\centering  
\begin{tabular}{rrcccccc}    
  \hline 
	  \hline     
		\multicolumn{1}{r}{} & \multicolumn{3}{c}{First-type proxy-measure}	&\multicolumn{1}{r}{} &	 \multicolumn{3}{c}{Second-type proxy-measure}\\          
	\cline{2-4}  	\cline{6-8}              
    & $GSI^F_{sup,u}$  & $U_u$ & $\widehat{U\!B_u^F}$ & & $GSI^{l2}_{sup,u}$ & $U_u^{l2}$ & $\widehat{U\!B_u^{l2}}$ \\     
 \cline{2-4}  	\cline{6-8}     
  $X_1$ & 0.231 & 0.329 & 0.271 (0.004) && 0.231 & 0.329 & 0.271 (0.004) \\    
  $X_2$ & 0.214 & 0.285 & 0.238 (0.003) && 0.214 & 0.285 & 0.238 (0.003) \\ 
	$X_3$ & 0.196 & 0.272 & 0.223 (0.003) && 0.196 & 0.272 & 0.223 (0.003) \\    
  $X_4$ & 0.176 & 0.231 & 0.192 (0.002) && 0.176 & 0.231 & 0.192 (0.002) \\ 
	$X_5$ & 0.231 & 0.329 & 0.270 (0.004) && 0.231 & 0.329 & 0.270 (0.004) \\    
  $X_6$ & 0.256 & 0.345 & 0.291 (0.004) && 0.256 & 0.345 & 0.291 (0.004)  \\  
  \hline           
\multicolumn{2}{r}{} & \multicolumn{2}{c}{} 	&\multicolumn{1}{r}{} &	\multicolumn{1}{r}{}	&	 \multicolumn{2}{c}{} \\                                     
   \cline{2-4}  	\cline{6-8}       
$X_1:X_2$ &  0     & 0 & 1.3e-19 (3e-19) && 0     & 0 & 1.3e-19 (3e-19) \\   
$X_1:X_3$ &  0.067 & 0.133 & 0.089 (0.001) && 0.067 & 0.133 & 0.089 (0.001) \\   
$X_1:X_4$ &  0     & 0 & 1.9e-19 (6e-19) && 0     & 0 & 1.9e-19 (6e-19) \\   
$X_1:X_5$ &  0.078 & 0.161 & 0.108 (0.001) && 0.078 & 0.161 & 0.108 (0.001) \\   
$X_1:X_6$ &  0     & 0 & 1.8e-18 (9e-18) && 0     & 0 & 1.8e-18 (9e-18)   \\ 
$X_2:X_3$ &  0     & 0 & 3e-18 (e-17) &&  0     & 0 & 3e-18 (e-17) \\   
$X_2:X_4$ &  0.040 & 0.085 & 0.054 (0.001) &&  0.040 & 0.085 & 0.054 (0.001) \\   
$X_2:X_5$ &  0     & 0 & 2.6e-19 (6e-19) &&  0     & 0 & 2.6e-19 (6e-19) \\   
$X_2:X_6$ &  0.053 & 0.127 & 0.080 (0.001) &&  0.053 & 0.127 & 0.080 (0.001)\\   
$X_3:X_4$ &  0     & 0 & 3e-19 (e-18) &&  0     & 0 & 3e-19 (e-18) \\   
$X_3:X_5$ &  0.067 & 0.133 & 0.089 (0.001) &&  0.067 & 0.133 & 0.089 (0.001) \\   
$X_3:X_6$ &  0     & 0 & 1.5e-18 (5e-18) &&   0     & 0 & 1.5e-18 (5e-18) \\     
$X_4:X_5$ &  0     & 0 & 4.4e-19 (9e-19) && 0     & 0 & 4.4e-19 (9e-19)  \\
$X_4:X_6$ &  0.046 & 0.103 & 0.065 (0.001) &&  0.046 & 0.103 & 0.065 (0.001) \\   
$X_5:X_6$ &  0     & 0 & 7.5e-19 (2e-18) &&  0     & 0 & 7.5e-19 (2e-18) \\         
  \hline                           
	  \hline                                             
\end{tabular}                    
\caption{True Sobol' indices and the estimates of the proxy-measure of these indices (average over $30$ replications, followed by their standard deviations in bracket) for the block-additive function. We also added the best (known) Poincar\'e upper bound $U_u$ (\cite{lamboni13,roustant14}).   While, the top part of this table focuses on the total indices, the bottom part deals with the total-interaction indices. For concise reporting of small values, we use e-a for $10^{-a}$.}
\label{tab:blockso}    
\end{table}

\subsubsection{Multivariate Sobol's function ($d=10$, $\NN=4$)}
The multivariate Sobol function includes 10 independent input factors following  a uniform distribution on $[0,\, 1]$ (\cite{lamboni18a}). It is defined as follows:
\begin{equation} \label{eq:gsobol}       
\M(\bo{\x}) =  \left[ \begin{array}{c} 
\prod_{j=1}^{d=10}\frac{ |4\, \x_j \,- \,2| \,+ \,\mathcal{A}[1,j]}{1 \,+\, \mathcal{A}[1,j]} \\ 
\prod_{j=1}^{d=10}\frac{ |4\, \x_j \,-\, 2| \,+ \,\mathcal{A}[2,j]}{1 \,+ \,\mathcal{A}[2,j]} \\ 
\prod_{j=1}^{d=10}\frac{ |4 \, \x_j \,- \,2|\, + \,\mathcal{A}[3,j]}{1 \,+\, \mathcal{A}[3,j]} \\ 
\prod_{j=1}^{d=10}\frac{ |4 \, \x_j \,- \,2| \, + \,\mathcal{A}[4,j]}{1 \,+\, \mathcal{A}[4,j]} 
\end{array}     
\right]    \, .          
\end{equation}          
                      
According to the values of $\mathcal{A}$ (matrix of type $4\times d$), this function has different properties.                  
If $$\mathcal{A} =\left[ \begin{array}{cccccccccc}
0 & 0 & 6.52 & 6.52 & 6.52 & 6.52 & 6.52 & 6.52 & 6.52 & 6.52 \\
0 & 1 & 4.5 & 9 & 99 & 99 & 99 & 99 & 99 & 99 \\ 
1 & 2 & 3 & 4 & 5 & 6 & 7 & 8 & 9 & 10 \\ 
50 & 50 & 50 & 50 & 50 & 50 & 50 & 50 & 50 & 50
\end{array}  \right]\, , $$ 
this function belongs to the class of functions having few important input factors. We can see that this function is continuous but it is differentiable almost everywhere.   
The true values of the GSIs and the estimates of the proxy-measure are listed in Table \ref{tab:sobA}.          
        
\begin{table}[ht]   
\centering  
\begin{tabular}{rrcccccc}
  \hline 
	  \hline     
		\multicolumn{1}{r}{} & \multicolumn{3}{c}{First-type proxy-measure}	&\multicolumn{1}{r}{} &	 \multicolumn{3}{c}{Second-type proxy-measure}\\       
	\cline{2-4}  	\cline{6-8}              
    & $GSI^F_{sup,j}$  & $U_j$ & $\widehat{U\!B_j^F}$ & & $GSI^{l2}_{sup,j}$ & $U_j^{l2}$ & $\widehat{U\!B_j^{l2}}$ \\     
 \cline{2-4}  	\cline{6-8}        
  $X_1$ & 0.605 & - & 2.419 (0.03) && 0.147 & - & 0.589 (0.007)   \\    
  $X_2$ & 0.406 & - & 1.626 (0.02) && 0.103 & - & 0.403 (0.005) \\ 
	$X_3$ & 0.034 & - & 0.134 (0.001)&& 0.008 & - & 0.032 (0.000)  \\    
  $X_4$ & 0.021 & - & 0.082 (0.001)&& 0.005 & - & 0.019 (0.000) \\ 
	$X_5$ & 0.014 & - & 0.057 (0.001) && 0.003 & - & 0.013 (0.000) \\    
  $X_6$ &  0.013 & - & 0.050 (0.001) &&0.003  & - & 0.012 (0.000) \\
	$X_7$ &  0.011 & - & 0.045 (0.000)&&0.003  & - & 0.010 (0.000) \\
	$X_8$ &  0.011 & - & 0.042 (0.001)&& 0.002 & - & 0.010 (0.000) \\
	$X_9$ &  0.010 & - & 0.039 (0.000)&& 0.002 & - & 0.009 (0.000) \\
	$X_{10}$ &  0.009 & - & 0.038 (0.001)&& 0.002 & - &0.009 (0.000) \\
  \hline                                                          
	   \hline                                                                      
\end{tabular}         
\caption{True total generalized sensitivity indices and the estimates of the proxy-measure of these indices (average over $30$ replications, followed by their standard deviations in bracket) for the multivariate Sobol function of type A.}    
\label{tab:sobA}        
\end{table}                                              
                
\subsection{Numerical results and discussion}  \label{sec:num2}     
Tables \ref{tab:ishi}-\ref{tab:sobA} report the estimated values of the proxy-measure of GSIs for different functions considered in this paper.   
From Tables \ref{tab:ishi}-\ref{tab:sobA}, it comes out that the proxy-measure values of GSIs and Sobol' indices are all the upper bounds of these indices, and they allow for identifying the important input factors and interaction among these inputs. While the prioritization of input factors using the proxy-measure is similar to the classification based on the GSIs for the block-additive function and the multivariate Sobol function; we note some difference in the case of Ishigami' functions (Tables \ref{tab:ishi} and \ref{tab:ishim}). Indeed, the proxy-measure identifies $\X_2$ as the most important input factor for both Ishigami'  functions. \\ 
                
For the Ishigami function in Table \ref{tab:ishi}, it appears that our proxy-measure (upper bound) improves the classical upper bound ($U_u$) from \cite{lamboni13, roustant14}. We obtained similar results for the block-additive function (Table \ref{tab:blockso}). For the block-additive function, our proxy-measure values are very close to the total and the total-interaction indices.\\
                  
 The second-type  proxy-measure of GSIs, which is equivalent to the first-type ones in the case of single-response models ($\NN=1$), provides the same information but with different proxy-measure values in the case of the multivariate response models ($\NN >1$). Although both proxy-measures provide the same information for the functions considered in this paper, the second-type proxy-measure accounts for the correlation among the components of the total-effect and total-interaction functionals, and it should be preferred in practice.      
        
\section{Conclusion} \label{sec:con}  
First, we propose two derivative-based expressions of the variance and new integral inequalities
for any function and for a wide class of Borel probability measures. For a function $\R \to \R^\NN$, we provide the optimal weighted Poincar\'e-type inequality. The proposed expressions of the variance and inequalities make use of the gradient and/or cross-partial derivatives, cumulative distribution functions, and probability density functions. They are based on a novel decomposition of multivariate functions developed in this paper. Second, the new weighted Poincar\'e-type inequalities are used to establish a new proxy-measure of the generalized sensitivity indices from multivariate sensitivity analysis, including Sobol' indices. The new proxy-measure of GSIs is an upper bound of the total or total-interaction GSIs, and we have shown that this new proxy-measure coincides with the total and total-interaction GSIs for a specific class of functions. For this class of functions, the new proxy-measure will give the right ranking of input factors using few model evaluations. Third, we construct unbiased and consistent estimators of both types of the proxy-measure.\\  
                                                                        
The numerical tests confirmed that our proxy-measure of the total and total-interaction GSIs improves the classical upper bound from the Poincar\'e-type inequalities (\cite{kucherenko09,lamboni13,roustant14}) for the functions considered in this paper. In the next future, it is interesting to investigate i) the integral equality and inequalities in presence of dependent input factors, ii) the proxy-measure in some given directions such as the directions driven by the model gradient.   
                                                 
\section*{Acknowledgements}
We would like to thank the two referees for their comments and suggestions that 
have helped improving this paper.


\begin{thebibliography}{38}

\bibitem{bobkov99}
\begin{barticle}[author]
\bauthor{\bsnm{Bobkov},~\bfnm{S.~G.}\binits{S.~G.}}
(\byear{1999}).
\btitle{Isoperimetric and Analytic Inequalities for Log-Concave Probability
  Measures}.
\bjournal{The Annals of Probability}
\bvolume{27}
\bpages{1903-1921}.
\end{barticle}
\endbibitem

\bibitem{bobkov09}
\begin{barticle}[author]
\bauthor{\bsnm{Bobkov},~\bfnm{Sergey~G.}\binits{S.~G.}} \AND
  \bauthor{\bsnm{Ledoux},~\bfnm{Michel}\binits{M.}}
(\byear{2009}).
\btitle{Weighted Poincaré-type inequalities for Cauchy and other convex
  measures}.
\bjournal{Ann. Probab.}
\bvolume{37}
\bpages{403--427}.
\end{barticle}
\endbibitem

\bibitem{bobkov09a}
\begin{binbook}[author]
\bauthor{\bsnm{Bobkov},~\bfnm{Sergey~G.}\binits{S.~G.}} \AND
  \bauthor{\bsnm{Ledoux},~\bfnm{Michel}\binits{M.}}
(\byear{2009}).
\btitle{On weighted isoperimetric and Poincaré-type inequalities}.
In \bbooktitle{High Dimensional Probability V: The Luminy Volume}.
\bseries{Collections}
\bvolume{Volume 5}
\bpages{1--29}.
\bpublisher{Institute of Mathematical Statistics}, \baddress{Beachwood, Ohio,
  USA}.
\end{binbook}
\endbibitem

\bibitem{borgonovo14}
\begin{barticle}[author]
\bauthor{\bsnm{Borgonovo},~\bfnm{E.}\binits{E.}},
  \bauthor{\bsnm{Tarantola},~\bfnm{S.}\binits{S.}},
  \bauthor{\bsnm{Plischke},~\bfnm{E.}\binits{E.}} \AND
  \bauthor{\bsnm{Morris},~\bfnm{M.~D.}\binits{M.~D.}}
(\byear{2014}).
\btitle{Transformations and invariance in the sensitivity analysis of computer
  experiments}.
\bjournal{Journal of the Royal Statistical Society Series B}
\bvolume{76}
\bpages{925-947}.
\end{barticle}
\endbibitem

\bibitem{courant36}
\begin{bbook}[author]
\bauthor{\bsnm{Courant},~\bfnm{R.}\binits{R.}}
(\byear{1936}).
\btitle{Differential and integral calculus}
\bvolume{II}.
\bpublisher{Blackie and Son Limited, Great Britain}.
\end{bbook}
\endbibitem

\bibitem{dutang13}
\begin{bmanual}[author]
\bauthor{\bsnm{Dutang},~\bfnm{Christophe}\binits{C.}} \AND
  \bauthor{\bsnm{Savicky},~\bfnm{Petr}\binits{P.}}
(\byear{2013}).
\btitle{randtoolbox: Generating and Testing Random Numbers}
\bnote{{\sc R} package version 1.13}.
\end{bmanual}
\endbibitem

\bibitem{efron81}
\begin{barticle}[author]
\bauthor{\bsnm{Efron},~\bfnm{B.}\binits{B.}} \AND
  \bauthor{\bsnm{Stein},~\bfnm{C.}\binits{C.}}
(\byear{1981}).
\btitle{The jacknife estimate of variance}.
\bjournal{The Annals of Statistics}
\bvolume{9}
\bpages{586-596}.
\end{barticle}
\endbibitem

\bibitem{fruth14}
\begin{barticle}[author]
\bauthor{\bsnm{Fruth},~\bfnm{J.}\binits{J.}},
  \bauthor{\bsnm{Roustant},~\bfnm{O.}\binits{O.}} \AND
  \bauthor{\bsnm{Kuhnt},~\bfnm{S.}\binits{S.}}
(\byear{2014}).
\btitle{Total interaction index: A variance-based sensitivity index for
  second-order interaction screening}.
\bjournal{Journal of Statistical Planning and Inference}
\bvolume{147}
\bpages{212 - 223}.
\end{barticle}
\endbibitem

\bibitem{gamboa14}
\begin{barticle}[author]
\bauthor{\bsnm{Gamboa},~\bfnm{Fabrice}\binits{F.}},
  \bauthor{\bsnm{Janon},~\bfnm{Alexandre}\binits{A.}},
  \bauthor{\bsnm{Klein},~\bfnm{Thierry}\binits{T.}} \AND
  \bauthor{\bsnm{Lagnoux},~\bfnm{Agn{\`e}s}\binits{A.}}
(\byear{2014}).
\btitle{Sensitivity indices for multivariate outputs}.
\bjournal{Comptes Rendus de l'Académie des Sciences}
\bpages{In press}.
\end{barticle}
\endbibitem

\bibitem{ghanem17}
\begin{bbook}[author]
\bauthor{\bsnm{Ghanem},~\bfnm{R.}\binits{R.}},
  \bauthor{\bsnm{Higdon},~\bfnm{D.}\binits{D.}} \AND
  \bauthor{\bsnm{Owhadi},~\bfnm{H.}\binits{H.}}
(\byear{2017}).
\btitle{Handbook of Uncertainty Quantification}.
\bpublisher{Springer International Publishing}.
\end{bbook}
\endbibitem

\bibitem{hardy73}
\begin{bbook}[author]
\bauthor{\bsnm{Hardy},~\bfnm{G.~H.}\binits{G.~H.}},
  \bauthor{\bsnm{Littlewood},~\bfnm{J.~E.}\binits{J.~E.}} \AND
  \bauthor{\bsnm{Polya},~\bfnm{G.}\binits{G.}}
(\byear{1973}).
\btitle{Inequalities}.
\bpublisher{Cambridge University Press, Second edition}.
\end{bbook}
\endbibitem

\bibitem{hoeffding48a}
\begin{barticle}[author]
\bauthor{\bsnm{Hoeffding},~\bfnm{W.}\binits{W.}}
(\byear{1948}).
\btitle{A class of statistics with asymptotically normal distribution}.
\bjournal{Annals of Mathematical Statistics}
\bvolume{19}
\bpages{293--325}.
\end{barticle}
\endbibitem

\bibitem{kubicek15}
\begin{barticle}[author]
\bauthor{\bsnm{Kubicek},~\bfnm{Martin}\binits{M.}},
  \bauthor{\bsnm{Minisci},~\bfnm{Edmondo}\binits{E.}} \AND
  \bauthor{\bsnm{Cisternino},~\bfnm{Marco}\binits{M.}}
(\byear{2015}).
\btitle{High dimensional sensitivity analysis using surrogate modeling and high
  dimensional model representation}.
\bjournal{International Journal for Uncertainty Quantification}
\bvolume{5}
\bpages{393--414}.
\end{barticle}
\endbibitem

\bibitem{kucherenko09}
\begin{barticle}[author]
\bauthor{\bsnm{Kucherenko},~\bfnm{S.}\binits{S.}},
  \bauthor{\bsnm{Rodriguez-Fernandez},~\bfnm{M.}\binits{M.}},
  \bauthor{\bsnm{Pantelides},~\bfnm{C.}\binits{C.}} \AND
  \bauthor{\bsnm{Shah},~\bfnm{N.}\binits{N.}}
(\byear{2009}).
\btitle{{Monte Carlo} evaluation of derivative-based global sensitivity
  measures}.
\bjournal{Reliability Engineering and System Safety}
\bvolume{94}
\bpages{1135-1148}.
\end{barticle}
\endbibitem

\bibitem{kucherenko16}
\begin{barticle}[author]
\bauthor{\bsnm{Kucherenko},~\bfnm{S.}\binits{S.}} \AND
  \bauthor{\bsnm{Song},~\bfnm{S.}\binits{S.}}
(\byear{2016}).
\btitle{Derivative-Based Global Sensitivity Measures and Their Link with Sobol'
  Sensitivity Indices. In: {C}ools {R}., {N}uyens {D}. (eds) {M}onte {C}arlo
  and {Q}uasi-{M}onte {C}arlo Methods}.
\bjournal{Springer Proceedings in Mathematics \& Statistics}
\bvolume{163}.
\bdoi{https://doi.org/10.1007/978-3-319-33507-0\_23}
\end{barticle}
\endbibitem

\bibitem{lamboni16b}
\begin{barticle}[author]
\bauthor{\bsnm{Lamboni},~\bfnm{M.}\binits{M.}}
(\byear{2016}).
\btitle{Global sensitivity analysis: a generalized, unbiased and optimal
  estimator of total-effect variance}.
\bjournal{Statistical Papers}
\bpages{1--26}.
\bdoi{10.1007/s00362-016-0768-5}
\end{barticle}
\endbibitem

\bibitem{lamboni16}
\begin{barticle}[author]
\bauthor{\bsnm{Lamboni},~\bfnm{M.}\binits{M.}}
(\byear{2016}).
\btitle{Global sensitivity analysis: an efficient numerical method for
  approximating the total sensitivity index}.
\bjournal{International Journal for Uncertainty Quantification}
\bvolume{6}
\bpages{1-17}.
\end{barticle}
\endbibitem

\bibitem{lamboni18a}
\begin{barticle}[author]
\bauthor{\bsnm{Lamboni},~\bfnm{Matieyendou}\binits{M.}}
(\byear{2018}).
\btitle{Multivariate sensitivity analysis: minimum variance unbiased estimators
  of the first-order and total-effect covariance matrices}.
\bjournal{Reliability Engineering \& System Safety}
\bpages{-}.
\bdoi{https://doi.org/10.1016/j.ress.2018.06.004}
\end{barticle}
\endbibitem

\bibitem{lamboni18}
\begin{barticle}[author]
\bauthor{\bsnm{Lamboni},~\bfnm{Matieyendou}\binits{M.}}
(\byear{2018}).
\btitle{Uncertainty quantification: a minimum variance unbiased (joint)
  estimator of the non-normalized Sobol' indices}.
\bjournal{Stat Papers}
\bpages{-}.
\bdoi{https://doi.org/10.1007/s00362-018-1010-4}
\end{barticle}
\endbibitem

\bibitem{lamboni19}
\begin{barticle}[author]
\bauthor{\bsnm{Lamboni},~\bfnm{M.}\binits{M.}}
(\byear{2019}).
\btitle{Derivative-based generalized sensitivity indices and Sobol' indices}.
\bjournal{Submitted to MCS}
\bvolume{-}
\bpages{-}.
\end{barticle}
\endbibitem

\bibitem{lamboni13}
\begin{barticle}[author]
\bauthor{\bsnm{Lamboni},~\bfnm{M.}\binits{M.}},
  \bauthor{\bsnm{Iooss},~\bfnm{B.}\binits{B.}},
  \bauthor{\bsnm{Popelin},~\bfnm{A.~L.}\binits{A.~L.}} \AND
  \bauthor{\bsnm{Gamboa},~\bfnm{F.}\binits{F.}}
(\byear{2013}).
\btitle{Derivative-based global sensitivity measures: General links with
  {S}obol' indices and numerical tests}.
\bjournal{Mathematics and Computers in Simulation}
\bvolume{87}
\bpages{45 - 54}.
\end{barticle}
\endbibitem

\bibitem{lamboni09}
\begin{barticle}[author]
\bauthor{\bsnm{Lamboni},~\bfnm{M.}\binits{M.}},
  \bauthor{\bsnm{Makowski},~\bfnm{D.}\binits{D.}},
  \bauthor{\bsnm{Lehuger},~\bfnm{S.}\binits{S.}},
  \bauthor{\bsnm{Gabrielle},~\bfnm{B.}\binits{B.}} \AND
  \bauthor{\bsnm{Monod},~\bfnm{H.}\binits{H.}}
(\byear{2009}).
\btitle{Multivariate global sensitivity analysis for dynamic crop models}.
\bjournal{Fields Crop Reasearch}
\bvolume{113}
\bpages{312-320}.
\end{barticle}
\endbibitem

\bibitem{lamboni11}
\begin{barticle}[author]
\bauthor{\bsnm{Lamboni},~\bfnm{M.}\binits{M.}},
  \bauthor{\bsnm{Monod},~\bfnm{H.}\binits{H.}} \AND
  \bauthor{\bsnm{Makowski},~\bfnm{D.}\binits{D.}}
(\byear{2011}).
\btitle{Multivariate sensitivity analysis to measure global contribution of
  input factors in dynamic models}.
\bjournal{Reliability Engineering and System Safety}
\bvolume{96}
\bpages{450-459}.
\end{barticle}
\endbibitem

\bibitem{liu06}
\begin{barticle}[author]
\bauthor{\bsnm{Liu},~\bfnm{Ruixue}\binits{R.}} \AND
  \bauthor{\bsnm{Owen},~\bfnm{Art~B}\binits{A.~B.}}
(\byear{2006}).
\btitle{Estimating Mean Dimensionality of Analysis of Variance Decompositions}.
\bjournal{Journal of the American Statistical Association}
\bvolume{101}
\bpages{712-721}.
\end{barticle}
\endbibitem

\bibitem{morris91}
\begin{barticle}[author]
\bauthor{\bsnm{Morris},~\bfnm{M.~D.}\binits{M.~D.}}
(\byear{1991}).
\btitle{Factorial sampling plans for preliminary computational experiments}.
\bjournal{Technometrics}
\bvolume{33}
\bpages{161-174}.
\end{barticle}
\endbibitem

\bibitem{owen13b}
\begin{barticle}[author]
\bauthor{\bsnm{Owen},~\bfnm{A.~B.}\binits{A.~B.}}
(\byear{2013}).
\btitle{Variance Components and Generalized {S}obol' Indices}.
\bjournal{SIAM/ASA Journal on Uncertainty Quantification}
\bvolume{1}
\bpages{19-41}.
\end{barticle}
\endbibitem

\bibitem{owen13}
\begin{barticle}[author]
\bauthor{\bsnm{Owen},~\bfnm{Art~B.}\binits{A.~B.}}
(\byear{2013}).
\btitle{Better estimation of small {S}obol' sensitivity indices}.
\bjournal{ACM Trans. Model. Comput. Simul.}
\bvolume{23}
\bpages{111-1117}.
\end{barticle}
\endbibitem

\bibitem{plischke13}
\begin{barticle}[author]
\bauthor{\bsnm{Plischke},~\bfnm{Elmar}\binits{E.}},
  \bauthor{\bsnm{Borgonovo},~\bfnm{Emanuele}\binits{E.}} \AND
  \bauthor{\bsnm{Smith},~\bfnm{Curtis~L.}\binits{C.~L.}}
(\byear{2013}).
\btitle{Global sensitivity measures from given data}.
\bjournal{European Journal of Operational Research}
\bvolume{226}
\bpages{536 - 550}.
\end{barticle}
\endbibitem

\bibitem{ricciardi05}
\begin{barticle}[author]
\bauthor{\bsnm{Ricciardi},~\bfnm{T.}\binits{T.}}
(\byear{2005}).
\btitle{A sharp weighted Wirtinger inequality}.
\bjournal{Bolletino . U.M.I.}
\bvolume{8}
\bpages{259-267}.
\end{barticle}
\endbibitem

\bibitem{roustant17}
\begin{barticle}[author]
\bauthor{\bsnm{Roustant},~\bfnm{Olivier}\binits{O.}},
  \bauthor{\bsnm{Barthe},~\bfnm{Franck}\binits{F.}} \AND
  \bauthor{\bsnm{Iooss},~\bfnm{Bertrand}\binits{B.}}
(\byear{2017}).
\btitle{Poincaré inequalities on intervals  application to sensitivity
  analysis}.
\bjournal{Electron. J. Statist.}
\bvolume{11}
\bpages{3081--3119}.
\end{barticle}
\endbibitem

\bibitem{roustant14}
\begin{barticle}[author]
\bauthor{\bsnm{Roustant},~\bfnm{O.}\binits{O.}},
  \bauthor{\bsnm{Fruth},~\bfnm{J.}\binits{J.}},
  \bauthor{\bsnm{Iooss},~\bfnm{B.}\binits{B.}} \AND
  \bauthor{\bsnm{Kuhnt},~\bfnm{S.}\binits{S.}}
(\byear{2014}).
\btitle{Crossed-derivative based sensitivity measures for interaction
  screening}.
\bjournal{Mathematics and Computers in Simulation}
\bvolume{105}
\bpages{105 - 118}.
\end{barticle}
\endbibitem

\bibitem{saltelli02b}
\begin{barticle}[author]
\bauthor{\bsnm{Saltelli},~\bfnm{A.}\binits{A.}}
(\byear{2002}).
\btitle{Making best use of model evaluations to compute sensitivity indices}.
\bjournal{Computer Physics Communications}
\bvolume{145}
\bpages{280-297}.
\end{barticle}
\endbibitem

\bibitem{saltelli10b}
\begin{barticle}[author]
\bauthor{\bsnm{Saltelli},~\bfnm{A.}\binits{A.}},
  \bauthor{\bsnm{Annoni},~\bfnm{P.}\binits{P.}},
  \bauthor{\bsnm{Azzini},~\bfnm{I.}\binits{I.}},
  \bauthor{\bsnm{Campolongo},~\bfnm{F.}\binits{F.}},
  \bauthor{\bsnm{Ratto},~\bfnm{M.}\binits{M.}} \AND
  \bauthor{\bsnm{Tarantola},~\bfnm{S.}\binits{S.}}
(\byear{2010}).
\btitle{Variance based sensitivity analysis of model output. {D}esign and
  estimator for the total sensitivity index}.
\bjournal{Computer Physics Communications}
\bvolume{181}
\bpages{259-270}.
\end{barticle}
\endbibitem

\bibitem{saltelli00}
\begin{bbook}[author]
\bauthor{\bsnm{Saltelli},~\bfnm{A}\binits{A.}},
  \bauthor{\bsnm{Chan},~\bfnm{K.}\binits{K.}} \AND
  \bauthor{\bsnm{Scott},~\bfnm{E.}\binits{E.}}
(\byear{2000}).
\btitle{Variance-Based Methods}.
\bseries{Probability and Statistics}.
\bpublisher{John Wiley and Sons}.
\end{bbook}
\endbibitem

\bibitem{sobol93}
\begin{barticle}[author]
\bauthor{\bsnm{Sobol},~\bfnm{I.~M.}\binits{I.~M.}}
(\byear{1993}).
\btitle{Sensitivity analysis for non-linear mathematical models}.
\bjournal{Mathematical Modelling and Computational Experiments}
\bvolume{1}
\bpages{407-414}.
\end{barticle}
\endbibitem

\bibitem{sobol01}
\begin{barticle}[author]
\bauthor{\bsnm{Sobol},~\bfnm{I.~M.}\binits{I.~M.}}
(\byear{2001}).
\btitle{Global sensitivity indices for nonlinear mathematical models and their
  {Monte Carlo} estimates}.
\bjournal{Mathematics and Computers in Simulation}
\bvolume{55}
\bpages{271 - 280}.
\end{barticle}
\endbibitem

\bibitem{sobol09}
\begin{barticle}[author]
\bauthor{\bsnm{Sobol},~\bfnm{I.~M.}\binits{I.~M.}} \AND
  \bauthor{\bsnm{Kucherenko},~\bfnm{S.}\binits{S.}}
(\byear{2009}).
\btitle{Derivative based global sensitivity measures and the link with global
  sensitivity indices}.
\bjournal{Mathematics and Computers in Simulation}
\bvolume{79}
\bpages{3009-3017}.
\end{barticle}
\endbibitem

\bibitem{xiao17}
\begin{barticle}[author]
\bauthor{\bsnm{Xiao},~\bfnm{Sinan}\binits{S.}},
  \bauthor{\bsnm{Lu},~\bfnm{Zhenzhou}\binits{Z.}} \AND
  \bauthor{\bsnm{Xu},~\bfnm{Liyang}\binits{L.}}
(\byear{2017}).
\btitle{Multivariate sensitivity analysis based on the direction of eigen space
  through principal component analysis}.
\bjournal{Reliability Engineering \& System Safety}
\bvolume{165}
\bpages{1 - 10}.
\end{barticle}
\endbibitem

\end{thebibliography}

\appendix
       
\section{Proof of Theorem \ref{lem:deriaov}}\label{app:deriaov}
For Point (i), when multiplying Equation (\ref{eq:ited}) by the joint PDF $\rho_{u}(\bo{\y}_{u})$, and integrating it over the joint support $\Omega_u$, we obtain   
\begin{eqnarray}   \label{eq:wited} 
 \M(\bo{\z}_{u}, \bo{\x}_{\sim u}) - \int_{\Omega_u} \M(\bo{\y}_{u}, \bo{\x}_{\sim u}) \, d\mu(\bo{\y}_{u})
&=& 
 \sum_{\substack{v,\, v \subseteq u \\ |v|>0}} \int_{\Omega_v} \int_{\bo{\y}_{v}}^{\bo{\z}_{v}}  
 \frac{\partial ^{|v|}\M}{\partial \bo{\x}_{v}}(\bo{x}) \rho_{v}(\bo{\y}_{v})\, d\bo{\x}_{v} d\bo{\y}_{v} \, ,                             
 \end{eqnarray}
      
Consider the measurable and differentiable function   $\bo{\rr} : \Omega_{u} \to [0,\, 1]^{|u|}$ given by $\rr_{j}(\bo{\x}_u) =\frac{\x_{j}-\y_{j}}{\z_{j}-\y_{j}}$ for $j =1,\, \ldots, \, |u|$. A change of variables gives    
                    
\begin{eqnarray}   \label{eq:wited2} 
 \M(\bo{\z}_{u}, \bo{\x}_{\sim u}) - \int_{\Omega_u} \M(\bo{\y}_{u}, \bo{\x}_{\sim u}) \, d\mu(\bo{\y}_{u})
&=& 
 \sum_{\substack{v,\, v \subseteq u \\ |v|>0}} \int_{\Omega_v} \int_{[0,\,1]^{|u|}}  
 \frac{\partial ^{|v|}\M}{\partial \bo{\x}_{v}}\left[(\rr_j (\z_j -\y_j) + \y_j, \, j \in v),\, \bo{\x}_{\sim v}\right]  \nonumber \\     
& & \times 
   \prod_{i \in v}\left(\z_i -\y_i \right) \rho_{v}(\bo{\y}_{v}) \, d\bo{\rr} d\bo{\y}_{v} \, .                             
 \end{eqnarray}

Replacing $\bo{\z}_{u}$ with $\bo{\x}_{u}$ in Equation (\ref{eq:wited2}) yields to  

\begin{eqnarray}   \label{eq:kubi32} 
 \M(\bo{\x}_{u}, \bo{\x}_{\sim u}) - \int_{\Omega_u} \M(\bo{\y}_{u}, \bo{\x}_{\sim u}) \, d\mu(\bo{\y}_{u})
&=& 
 \sum_{\substack{v,\, v \subseteq u \\ |v|>0}} \int_{\Omega_v} \int_{[0,\,1]^{|u|}}  
 \frac{\partial ^{|v|}\M}{\partial \bo{\x}_{v}}\left[(\rr_j (\x_j -\y_j) + \y_j, \, j \in v),\, \bo{\x}_{\sim v}\right]  \nonumber \\     
& & \times 
   \prod_{i \in v}\left(\x_i -\y_i \right) \rho_{v}(\bo{\y}_{v}) \, d\bo{\rr} d\bo{\y}_{v} \, .     \nonumber                          
 \end{eqnarray}                  

Consider the measurable and differentiable function $\bo{\ttt}:\, \Omega_u \to \Omega_u$ given by $\ttt_{j} (\bo{\y}) =\rr_{j} \x_{j} + \y_{j}(1-\rr_{j})$ for $j =1,\,2,\, \ldots, \, |u|$.  We can see that $\y_{j} = \frac{\ttt_{j} -\rr_{j} \x_{j}}{1-\rr_{j}}$, $d\y_{j} =\frac{d\ttt_{j}}{1-\rr_{j}}$, and $\x_{j}-\y_{j} = \frac{\x_{j}-\ttt_{j}}{1-\rr_{j}}$.   
 Bearing in mind Fubini-Lebesgue's theorem, a change of variables gives                  
\begin{eqnarray}   \label{eq:kubi4}  
\M(\bo{\x}_{u}, \bo{\x}_{\sim u}) - \int_{\Omega_u} \M(\bo{\y}_{u}, \bo{\x}_{\sim u}) \, d\mu(\bo{\y}_{u})
&=&      
 \sum_{\substack{v,\, v \subseteq u \\ |v|>0}} \int_{\Omega_v}  
 \frac{\partial ^{|v|}\M}{\partial \bo{\x}_{v}}\left(\bo{\ttt}_v,\, \bo{\x}_{\sim v}\right)  \nonumber \\     
& & \times      \int_{[0,\,1]^{|u|}} 
   \prod_{i \in v}\frac{\left(\x_{i}-\ttt_{i}\right)}{(1-\rr_{i})^2} \rho_{v}\left(\frac{\ttt_{j} -\rr_{j} \x_{j}}{1-\rr_{j}},\, j\in v\right) \, d\bo{\rr} d\bo{\ttt}_{v} \, . 
 \nonumber               
\end{eqnarray}    
                                
  Again, a change of variables of the form $\z_j(\bo{\rr})= \frac{\ttt_{j} -\rr_{j} \x_{j}}{1-\rr_{j}}$  (with $d\z_{j} =\frac{\ttt_{j}-\x_j}{(1-\rr_{j})^2} d\rr_{j}$, $\z_{j}(0)=\ttt_{j}$, and $\z_{j}(1)=\ttt_{j}\indic_{[\ttt_{j} =\x_j]} + \infty \indic_{[\ttt_{j} > \x_j]} -\infty \indic_{[\ttt_{j} <\x_j]} $) 	gives     
	   
\begin{eqnarray}   \label{eq:kubi5}  
\M(\bo{\x}_{u}, \bo{\x}_{\sim u}) - \int_{\Omega_u} \M(\bo{\y}_{u}, \bo{\x}_{\sim u}) \, d\mu(\bo{\y}_{u})
&=&  
 \sum_{\substack{v,\, v \subseteq u \\ |v|>0}}   \int_{\Omega_v}   
 \frac{\partial ^{|v|}\M}{\partial \bo{\x}_{v}}\left(\bo{\ttt}_v,\, \bo{\x}_{\sim v}\right)
 \int_{A_j,\, j \in v}  \nonumber   \\
& & \times  \prod_{i \in v}\left( \indic_{[\ttt_i < \x_i]} -\indic_{[\ttt_i \geq \x_i]} \right)  
    \rho_{v}\left( \bo{\z}_v \right) \, d\bo{\z}_v d\bo{\ttt}_{v} \, ,  \nonumber  
\end{eqnarray}       
with
$
A_j = \{\ttt_j\}\times \indic_{[\ttt_j =\x_j]} + [\ttt_j,\, +\infty [ \times  \indic_{[ \ttt_j  > \x_j]} + ]-\infty,\, \ttt_j] \times  \indic_{[ \ttt_j <  \x_j]}   
$.    \\              
		   
As the set $\{\ttt_j\}$ is negligible with respect to all probability measures that are absolutely continuous w.r.t. the Lebesgue measure, we consider only (in what follows)
$ 
A_j = [\ttt_j,\, +~\infty [ \times  \indic_{[ \ttt_j  \geq \x_j]} + ]-~\infty,\, \ttt_j] \times  \indic_{[ \ttt_j <  \x_j]}$ almost surely.	Bearing in mind the independent assumption (A1), we can write  

\begin{eqnarray} 
\int_{A_j,\, j \in v} \prod_{j \in v}	\left(\indic_{[ \ttt_j < \x_j]} -\indic_{[\ttt_j \geq \x_j]} \right) 
   \rho_{v}\left( \bo{\z}_v \right) \, d\bo{\z}_v 
 &=&               
 \prod_{j \in v}\int_{A_j}	\left( \indic_{[ \ttt_j < \x_j]} -\indic_{[\ttt_j \geq \x_j]} \right) 	\rho_j(\z_j)\, d\z_j \nonumber \\  
	&=&  \prod_{j \in v} \left(F_j (\ttt_{j}) - \indic_{[\ttt_{j} \geq \x_{j}]}\right) 
	\, ,    \nonumber                          
\end{eqnarray}                         
and the first results follows $\forall\, u\subseteq \{1,\,2,\, \ldots,\, d\}$ in general and for $u=\{1,\,2,\, \ldots,\, d\}$ in particular. \\   
      
For Point (ii), Equations (\ref{eq:center}) is obvious knowing that $\int_{\Omega_j} \left(F_j (\ttt_{j}) - \indic_{[\ttt_{j} \geq \x_{j}]}\right)\, d\mu(\x_j)  =0$. \\
    
For Point (iii), knowing that $\var(\M) =\int \M \M^\T \, d\mu-\int \M\, d\mu \int\M^\T\, d\mu$ and bearing in mind Fubini's theorem, we have         
\begin{eqnarray}     
\var(\gtef_v) &=& \int_{\Omega \times \Omega_v^2} 
 \frac{\partial^{|v|} \M}{\partial \bo{\x}_{v}}\left(\bo{\x}'_{v},\bo{\x}_{\sim v}\right) \frac{\partial^{|v|} \M^\T}{\partial \bo{\x}_{v}}\left(\bo{\z}'_{v},\bo{\x}_{\sim v}\right)  \prod_{j \in v} \frac{\left(F_{j}(\x_{j}') - \indic_{[\x_{j}' \geq \x_{j}]}\right)\left( F_{j}(\z_{j}') - \indic_{[\z_{j}' \geq \x_{j}]} \right)}{\rho_{j} (\x_{j}') \rho_{j} (\z_{j}') }  \nonumber \\
  & &  \times  d\mu(\bo{\x})  d\mu(\bo{\x}'_v)  d\mu(\bo{\z}'_v)    \nonumber \\
	&=& \int_{\Omega \times \Omega_v}   
 \frac{\partial^{|v|} \M}{\partial \bo{\x}_{v}}\left(\bo{\x}'_{v},\bo{\x}_{\sim v}\right) \frac{\partial^{|v|} \M^\T}{\partial \bo{\x}_{v}}\left(\bo{\z}'_{v},\bo{\x}_{\sim v}\right)  \prod_{j \in v} \frac{F_{j}\left(\min(\x_{j}',\,\z_{j}')\right) -F_{j}(\x_{j}')F_{j}(\z_{j}')}{\rho_{j} (\x_{j}')\rho_{j} (\z_{j}')} \nonumber \\
 & & \times  d\mu(\bo{\x}_{\sim v})  d\mu(\bo{\x}'_v)  d\mu(\bo{\z}'_v) \nonumber \\
	&=& \int_{\Omega \times \Omega_v} 
	\frac{\partial^{|v|} \M}{\partial \bo{\x}_{v}}\left(\bo{\x}\right) \frac{\partial^{|v|} \M^\T}{\partial \bo{\x}_{v}}\left(\bo{\x}'_{v},\bo{\x}_{\sim v}\right)  \prod_{j \in v} \frac{F_{j}\left(\min(\x_{j},\,\x_{j}')\right) -F_{j}(\x_{j})F_{j}(\x_{j}')}{\rho_{j} (\x_{j})\rho_{j} (\x_{j}')} \, d\mu(\bo{\x})  d\mu(\bo{\x}'_v) \, , \nonumber 
\end{eqnarray}               
as         
$$
\int_{\Omega_v} 
\prod_{k \in v} \frac{F_{k}(\x_{k}') - \indic_{[\x_{k}' \geq \x_{k}]}}{\rho_{k} (\x_{k}') }
 \frac{F_{k}(\z_{k}') - \indic_{[\z_{l}' \geq \x_{k}]}}{\rho_{k} (\z_{k}')} d\mu(\bo{\x}_{v})
= \prod_{k \in v} \frac{F_{k}\left(\min(\x_{k}',\,\z_{k}')\right) -F_{k}(\x_{k}')F_{k}(\z_{k}')}{\rho_{k} (\x_{k}')\rho_{k} (\z_{k}')}
\, .    
$$  

For Point (iv), using Equation (\ref{eq:deriaov}), we can write    
\begin{eqnarray}                                       
\var(\M)  & = & \sum_{\substack{v,\, v \subseteq \{1,\,\ldots,\, d\} \\ |v|>0}}
\sum_{\substack{w,\, w \subseteq \{1,\,\ldots,\, d\} \\ |w|>0}} \int_{\Omega \times \Omega_v\times \Omega_w}
 \frac{\partial^{|v|} \M}{\partial \bo{\x}_{v}}\left(\bo{\x}'_{v},\bo{\x}_{\sim v}\right) 
\frac{\partial^{|w|} \M^\T}{\partial \bo{\x}_{w}}\left(\bo{\z}'_{w},\bo{\x}_{\sim w}\right)   \nonumber \\             
& & \times
\prod_{j \in v} \frac{F_{j}(\x_{j}') - \indic_{[\x_{j}' \geq \x_{j}]}}{\rho_{j} (\x_{j}') }
 \prod_{k \in w} \frac{F_{k}(\z_{k}') - \indic_{[\z_{k}' \geq \x_{j}]}}{\rho_{k} (\z_{k}') }  \, d\mu(\bo{\x}'_v)d\mu(\bo{\z}'_w)d\mu(\bo{\x})  \nonumber \\ 
 & =&     
\sum_{\substack{v,\, v \subseteq \{1,\,\ldots,\, d\} \\ |v|>0}}
\int_{\Omega \times \Omega_v\times \Omega_w} 
	\frac{\partial^{|v|} \M}{\partial \bo{\x}_{v}}\left(\bo{\x}\right) \frac{\partial^{|v|} \M^\T}{\partial \bo{\x}_{v}}\left(\bo{\x}'_{v},\bo{\x}_{\sim v}\right)  \prod_{j \in v} \frac{F_{j}\left(\min(\x_{j},\,\x_{j}')\right) -F_{j}(\x_{j})F_{j}(\x_{j}')}{\rho_{j} (\x_{j})\rho_{j} (\x_{j}')} \, d\mu(\bo{\x})  d\mu(\bo{\x}'_v)   \nonumber \\ 
& &  +          
\sum_{\substack{v,\, v \subseteq \{1,\,\ldots,\, d\} \\ |v|>0}} 
\sum_{\substack{w,\, w \subseteq \{1,\,\ldots,\, d\} \\ w \neq v \\ |w|>0}} \int_{\Omega \times \Omega_v\times \Omega_w} 
\frac{\partial^{|v|} \M}{\partial \bo{\x}_{v}}\left(\bo{\x}'_{v},\bo{\x}_{\sim v}\right) 
\frac{\partial^{|w|} \M^\T}{\partial \bo{\x}_{w}}\left(\bo{\z}'_{w},\bo{\x}_{\sim w}\right)   \nonumber \\    
& & \times   
\prod_{j \in v} \frac{F_{j}(\x_{j}') - \indic_{[\x_{j}' \geq \x_{j}]}}{\rho_{j} (\x_{j}') }
 \prod_{k \in w} \frac{F_{k}(\z_{k}') - \indic_{[\z_{k}' \geq \x_{k}]}}{\rho_{k} (\z_{k}') }  \, d\mu(\bo{\x}'_v)d\mu(\bo{\z}'_w)d\mu(\bo{\x})   
 \, .  \nonumber       
\end{eqnarray}

\section{Proof of Theorem \ref{theo:deriaov1}}\label{app:deriaov1}  
The proof makes use of the well-known Sobol-Hoeffding's decomposition (\cite{hoeffding48a,efron81,sobol93}), that is, for a function $\M : \R^d \to \R^\NN$ satisfying assumptions (A1)-(A2), we have 
        
$$   \M(\bo{\x})   =  \M_\emptyset + \sum_{\substack{u,\,u \subseteq \{1,2, \ldots d\}\\|u|>0}} \M_u(\bo{\x}_u)\, ,      
$$   
 with                
$$\displaystyle \M_u(\bo{\x}_u) =  \sum_{w,\, w\subseteq u} (-1)^{|u|-|w|}\int_{\Omega_{\sim w}} \M(\bo{\x})\, d\mu(\bo{\x}_{\sim w})  \, .  
$$   
    
Using the derivative-based decomposition of $\M$ in  (\ref{eq:deriaov})
and knowing that $\sum_{w,\, w\subseteq u}  (-1)^{|u|-|w|} \M_{\emptyset} =\bo{0}$, we have  
\begin{eqnarray}  
\M_u(\bo{\x}_u)  &=&      
\sum_{w,\, w\subseteq u} \sum_{\substack{v,\, v \subseteq \{1,\,\ldots,\, d\}\\|v|>0}} (-1)^{|u|-|w|}  \int_{\Omega_{\sim w}}  \gtef_v(\bo{\x}) \, d\mu(\bo{\x}_{\sim w})  \, . 
 \end{eqnarray} 
     
From (\ref{eq:center}), $\int_{\Omega_{\sim w}}  \gtef_v(\bo{\x}) \, d\mu(\bo{\x}_{\sim w}) =0$ $\forall \, j \in v \cap \bar{w}$, or equivalently for the set $\{ v \subseteq \{1,\,\ldots,\, d\} :  v \cap \bar{w} \neq \emptyset \, \text{and} \,  |v|>0 \}$. Thus, the expectation of $\M_v(\bo{\x})$ with respect to $\bo{\x}_{\sim w}$ does not vanish for $\{v \subseteq \{1,\,\ldots,\, d\} : \, v \cap \bar{w} = \emptyset \, \text{and} \,  |v|>0 \} \equiv \{v\subseteq w :  |v|>0 \}$. Therefore, we can write 
\begin{eqnarray}  
\M_u(\bo{\x}_u)  &=&   
\sum_{w,\, w\subseteq u} \sum_{\substack{v,\, v \subseteq w \\|v|>0}} (-1)^{|u|-|w|}  \int_{\Omega_{\sim w}}  \gtef_v(\bo{\x}) \, d\mu(\bo{\x}_{\sim w})  \nonumber \\
 & = &  \sum_{w,\, w\subseteq u} \sum_{\substack{v,\, v \subseteq w \\|v|>0}} (-1)^{|u|-|w|}  \int_{\Omega_{\sim w} \times \Omega_v}   
 \frac{\partial^{|v|} \M}{\partial \bo{\x}_{v}}\left(\bo{\x}'_{v},\bo{\x}_{\sim v}\right)  \prod_{j \in v} \frac{F_{j}(\x_{j}') - \indic_{[\x_{j}' \geq \x_{j}]}}{\rho_{j} (\x_{j}') }  \, d\mu(\bo{\x}'_v) d\mu(\bo{\x}_{\sim w})  \nonumber \\
 &=&   \sum_{w,\, w\subseteq u} \sum_{\substack{v,\, v \subseteq w \\|v|>0}} (-1)^{|u|-|w|}  \int_{\Omega_{\sim w} \times \Omega_v}   
 \frac{\partial^{|v|} \M}{\partial \bo{\x}_{v}}\left(\bo{\x}'_{v}, \bo{\x}_{w \backslash v}, \bo{\x}_{\sim w}\right)  \prod_{j \in v} \frac{F_{j}(\x_{j}') - \indic_{[\x_{j}' \geq \x_{j}]}}{\rho_{j} (\x_{j}') }  \, d\mu(\bo{\x}'_v) d\mu(\bo{\x}_{\sim w})  \nonumber \\           
 &=&    \sum_{w,\, w\subseteq u} \sum_{\substack{v,\, v \subseteq w \\|v|>0}} (-1)^{|u|-|w|}  \int_{\Omega_{\sim w} \times \Omega_v}   
 \frac{\partial^{|v|} \M}{\partial \bo{\x}_{v}}\left(\bo{\x}'_{v}, \bo{\x}_{w \backslash v}, \bo{\x}_{\sim w}'\right)  \prod_{j \in v} \frac{F_{j}(\x_{j}') - \indic_{[\x_{j}' \geq \x_{j}]}}{\rho_{j} (\x_{j}') }  \, d\mu(\bo{\x}'_v) d\mu(\bo{\x}_{\sim w}')  \nonumber \\           
 &=&  \sum_{w,\, w\subseteq u} \sum_{\substack{v,\, v \subseteq w \\|v|>0}} (-1)^{|u|-|w|}  \int_{\Omega}   
 \frac{\partial^{|v|} \M}{\partial \bo{\x}_{v}}\left(\bo{\x}'_{v}, \bo{\x}_{w \backslash v}, \bo{\x}_{\sim w}'\right)  \prod_{j \in v} \frac{F_{j}(\x_{j}') - \indic_{[\x_{j}' \geq \x_{j}]}}{\rho_{j} (\x_{j}') }  \, d\mu(\bo{\x}')  \nonumber    
 \end{eqnarray}    
       
Points (ii) and Point (iii) are the consequences of the Hoeffding decomposition such as the orthogonality of $\M_u$.

\section{Proof of Lemma \ref{lem:ex}}\label{app:ex2}  
Without loss of generality, we suppose that $\boldsymbol{\alpha}=\bo{0}$. \\

The derivation of Point (i) is similar to a proof done in \cite{lamboni19}. We include that derivation for completeness reason and for interested readers. First, it is known that (\cite{lamboni19})          
$$ 
\tef_{u \backslash v}(\bo{\x})   = \sum_{\substack{v_1,\, v_1 \subseteq (u \backslash v) \\ |v_1|>0}} 
\int_{\Omega_u}   
   \frac{\partial^{|v_1|} \M}{\partial \bo{\x}_{v_1}}\left(\bo{\x}_{v_1}',\bo{\x}_{\sim v_1}\right) \prod_{j \in v_1} \frac{F_{j}(\x_{j}') - \indic_{[\x_{j}' \geq \x_{j}]} }{\rho_{j} (\x_{j}') } \, d\mu(\bo{\x}_{v_1}') \, . 
$$  

Second, bearing in mind the above expression of $\tef_{u \backslash v}(\bo{\x}) $, we can write  
\begin{eqnarray}                      
 \M_u^{sup}(\bo{\x}) &=& \sum_{\substack{v,\, v \subset u}}  
 (-1)^{|u|-|v| +1} \tef_{u \backslash v}(\bo{\x})  \nonumber \\   
 &=& \sum_{\substack{v,\, v \subset u}}      
 (-1)^{|u|-|v|+1}             
\sum_{\substack{v_1,\, v_1 \subseteq (u \backslash v) \\ |v_1|>0}} 
\int_{\Omega_u} 
   \frac{\partial^{|v_1|} \M}{\partial \bo{\x}_{v_1}}\left(\bo{\x}_{v_1}',\bo{\x}_{\sim v_1}\right) \prod_{j \in v_1} \frac{F_{j}(\x_{j}') - \indic_{[\x_{j}' \geq \x_{j}]} }{\rho_{j} (\x_{j}') } \, d\mu(\bo{\x}_{v_1}') \nonumber \\   
&=&  \sum_{\substack{v,\, v \subset u}} 
\sum_{\substack{v_1,\, v_1 \subseteq (u \backslash v) \\ |v_1|>0}} 
 (-1)^{|u|-|v|+1}   
\int_{\Omega_u} 
   \frac{\partial^{|v_1|} \M}{\partial \bo{\x}_{v_1}}\left(\bo{\x}_{v_1}',\bo{\x}_{\sim v_1}\right) \prod_{j \in v_1} \frac{F_{j}(\x_{j}') - \indic_{[\x_{j}' \geq \x_{j}]} }{\rho_{j} (\x_{j}') } \, d\mu(\bo{\x}_{v_1}')  
 \, .     \nonumber        
\end{eqnarray}           
For a given $v_1 \subseteq u$ with $v_1 \neq \emptyset$, we can see that $\left\{v,\, v_1\subseteq (u \backslash v) \right\} \equiv \left\{v=u \backslash t,\, v_1\subseteq t \subseteq u \right\}$. Thus, we have       
\begin{eqnarray}     
 \M_u^{sup}(\bo{\x}) &=& \sum_{\substack{v_1,\, v_1 \subseteq u \\ |v_1|>0}} 
\sum_{\substack{t,\, v_1 \subseteq t \subseteq u}}     
 (-1)^{|t|+1} \int_{\Omega_u} 
   \frac{\partial^{|v_1|} \M}{\partial \bo{\x}_{v_1}}\left(\bo{\x}_{v_1}',\bo{\x}_{\sim v_1}\right) \prod_{j \in v_1} \frac{F_{j}(\x_{j}') - \indic_{[\x_{j}' \geq \x_{j}]} }{\rho_{j} (\x_{j}') } \, d\mu(\bo{\x}_{v_1}')  
 \nonumber \\ 
&=& (-1)^{|u|+1} \int_{\Omega_u} 
   \frac{\partial^{|u|} \M}{\partial \bo{\x}_{u}}\left(\bo{\x}_{u}',\bo{\x}_{\sim u}\right) \prod_{j \in u} \frac{F_{j}(\x_{j}') - \indic_{[\x_{j}' \geq \x_{j}]} }{\rho_{j} (\x_{j}') } \, d\mu(\bo{\x}_{u}')  
   \, ,     \nonumber         
\end{eqnarray}                           
due to the fact that                      
\begin{eqnarray}   
\sum_{\substack{t,\, v_1 \subseteq t \subseteq u}}     
 (-1)^{|t|+1} &= &\sum_{\substack{|t| = |v_1|}}^{|u|} \binom{|u|-|v_1|}{|t|-|v_1|}(-1)^{|t|+1}  = \sum_{\substack{i = 0}}^{|u|-|v_1|} \binom{|u|-|v_1|}{i}(-1)^{i+|v_1|+1} \nonumber  \\
   &=& (-1)^{|u|+1} \indic_{\left[v_1 = u \right]}	   \, . \nonumber 
\end{eqnarray} 

For Point (ii), it is obvious that $\M_u^{sup}(\bo{x}) = \sum_{\substack{v,\, v \subset u}}  
 (-1)^{|u|-|v| +1} \tef_{u \backslash v}(\bo{\x})$ is differentiable w.r.t. $\bo{x}_u$, as it is the case of $\tef_{u \backslash v}$. Bearing in mind the property of $\tef_{u}(\bo{\x})$ in (\ref{eq:proptef}), we have   
\begin{eqnarray} 
\frac{\partial^{|u|} \M_u^{sup}}{\partial \bo{\x}_{u}}(\bo{\x}) &=& \sum_{\substack{v,\, v \subset u}}  
 (-1)^{|u|-|v| +1} \frac{\partial^{|u|} \tef_{u \backslash v}}{\partial \bo{\x}_{u}}(\bo{\x}) \nonumber \\        
 &=& \sum_{\substack{v,\, v \subset u}}  
 (-1)^{|u|-|v| +1} \frac{\partial^{|u|} \M}{\partial \bo{\x}_{u}}(\bo{\x}) \nonumber \\      
 &=&  \frac{\partial^{|u|} \M}{\partial \bo{\x}_{u}}(\bo{\x}) \, ,
\nonumber
\end{eqnarray}        
as  
$$              
\sum_{\substack{v,\, v \subset u}} (-1)^{|u|-|v| +1}
= (-1)^{|u|+1}\left(\sum_{k=0}^{|u|}\binom{|u|}{k}(-1)^k - (-1)^{|u|} \right) =1 \, .
$$       
The last result holds by replacing $\frac{\partial^{|u|} \M}{\partial \bo{\x}_{u}}(\bo{\x})$ with $\frac{\partial^{|u|} \M_u^{sup}}{\partial \bo{\x}_{u}}(\bo{\x})$. 
   
\section{Proof of Corollary \ref{coro:equa}}\label{app:equa}  
 Without loss of generality, we suppose that $\boldsymbol{\alpha}=\bo{0}$. \\

Point (i) is obvious, as it is a consequence of Theorem \ref{lem:deriaov}. \\

Point (ii) is obtained by replacing $\frac{\partial^{|u|} \M}{\partial \bo{\x}_{u}}(\bo{\x})$ with $\frac{\partial^{|u|} \tefg}{\partial \bo{\x}_{u}}(\bo{\x})$. \\

Bearing in mind Point (ii), assumptions (A1) and (A7) and the Fubini-Lebesgue theorem, Point (iii) is derived as follows: 
\begin{eqnarray}  
\var(\gtef)  & = &                                
\int_{\Omega\times \Omega_u} 
\frac{\partial^{|u|} \gtef}{\partial \bo{\x}_{u}}\left(\bo{\x}\right) 
\frac{\partial^{|u|} \gtef^\T}{\partial \bo{\x}_{u}}\left(\bo{\x}'_{u} ,\bo{\x}_{\sim u}\right) 
\prod_{k \in u} \frac{F_{k}\left[\min(\x_{k},\,\x_{k}')\right] -F_{k}(\x_{k})F_{k}(\x_{k}')}{\rho_{k} (\x_{k})\rho_{k} (\x_{k}')}  \, \nonumber \\ 
 &  &    \times        
d\mu(\bo{\x}) d\mu(\bo{\x}'_u)                                 
\,  ,  \nonumber    \\
 &= & \int_{\Omega\times \Omega_u} 
\frac{\partial^{|u|} \gtef}{\partial \bo{\x}_{u}}\left(\bo{\x}\right)
\frac{\partial^{|u|} \gtef^\T}{\partial \bo{\x}_{u}}\left(\bo{\x}\right) \frac{\rho_u(\bo{x}_u')}{\rho_u(\bo{x}_u)}  
\prod_{k \in u} \frac{F_{k}\left[\min(\x_{k},\,\x_{k}')\right] -F_{k}(\x_{k})F_{k}(\x_{k}')}{\rho_{k} (\x_{k})\rho_{k} (\x_{k}')}  \, \nonumber \\ 
 &  &    \times        
d\mu(\bo{\x}) d\mu(\bo{\x}'_u)                                 
\,  ,  \nonumber    \\
 &= & \int_{\Omega\times \Omega_u} 
\frac{\partial^{|u|} \gtef}{\partial \bo{\x}_{u}}\left(\bo{\x}\right)
\frac{\partial^{|u|} \gtef^\T}{\partial \bo{\x}_{u}}\left(\bo{\x}\right)
\prod_{k \in u} \frac{F_{k}\left[\min(\x_{k},\,\x_{k}')\right] -F_{k}(\x_{k})F_{k}(\x_{k}')}{\left[\rho_{k} (\x_{k})\right]^2} \, d\mu(\bo{\x}) d\mu(\bo{\x}'_u)         
\,  ,  \nonumber    \\    
 &=&    \frac{1}{2^{|u|}} 
\int 
\frac{\partial^{|u|} \gtef}{\partial \bo{\x}_{u}}\left(\bo{\x}\right) 
\frac{\partial^{|u|} \gtef^\T}{\partial \bo{\x}_{u}}\left(\bo{\x}\right) 
\prod_{k \in u} \frac{F_{k}(\x_{k})\left[1-F_{k}(\x_{k})\right]}{\left[\rho_{k} (\x_{k})\right]^2}  d\mu(\bo{\x}) \,  ,   \nonumber  
\end{eqnarray}    
as         
\begin{eqnarray}  
\int _{\Omega_k}
F_{k}\left[\min(\x_{k},\,\x_{k}')\right] -F_{k}(\x_{k})F_{k}(\x_{k}') d\mu(\bo{\x}_k') &=& \int_{-\infty}^{x_k} F_{k}(\x_{k}') \left[1 -F_{k}(\x_{k})\right] \rho(x_k') \, dx_k'\nonumber \\ 
   & & + \int_{x_k}^{+\infty} F_{k}(\x_{k}) \left[1 -F_{k}(\x_{k}')\right]  \rho(x_k') \, dx_k' \nonumber \\ 
	 &=& \frac{1}{2}\left[1 -F_{k}(\x_{k})\right]F_{k}(\x_{k})^2 \nonumber \\ 
	 && + F_{k}(\x_{k})\left(\frac{1}{2} - F_{k}(\x_{k}) + \frac{1}{2} F_{k}(\x_{k})^2\right) \nonumber \\       
	&=&    \frac{1}{2} F_{k}(\x_{k}) \left[1 -F_{k}(\x_{k})\right] \, . 
	\nonumber           
\end{eqnarray}             
                    
\section{Proof of Theorem \ref{theo:ineq}}\label{app:ineq}  
Firstly, knowing Equation (\ref{eq:decvtief1f}) and the fact that $\var(\gtef)$ is a  symmetric and positive semi-definite matrix (variance-covariance matrix), we can write    
\begin{eqnarray}
  \var(\gtef)&= &  \frac{1}{2}           
\int 
\frac{\partial^{|u|} \gtef}{\partial \bo{\x}_{u}}\left(\bo{\x}\right) 
\frac{\partial^{|u|} \gtef^\T}{\partial \bo{\x}_{u}}\left(\bo{\x}'_{u} ,\bo{\x}_{\sim u}\right) 
\prod_{k \in u} \frac{F_{k}\left[\min(\x_{k},\,\x_{k}')\right] -F_{k}(\x_{k})F_{k}(\x_{k}')}{\rho_{k} (\x_{k})\rho_{k} (\x_{k}')}  d\mu(\bo{\x}) d\mu(\bo{\x}')    \nonumber \\
&  & + \frac{1}{2}   
\int  
\frac{\partial^{|u|} \gtef}{\partial \bo{\x}_{u}}\left(\bo{\x}'_{u} ,\bo{\x}_{\sim u}\right)
\frac{\partial^{|u|} \gtef^\T}{\partial \bo{\x}_{u}}\left(\bo{\x}\right) 
\prod_{k \in u} \frac{F_{k}\left[\min(\x_{k},\,\x_{k}')\right] -F_{k}(\x_{k})F_{k}(\x_{k}')}{\rho_{k} (\x_{k})\rho_{k} (\x_{k}')}  d\mu(\bo{\x}) d\mu(\bo{\x}') \, .  \nonumber 
\end{eqnarray}   
 Secondly, let consider the matrix                
\begin{eqnarray}
W &:= & \int  
\frac{\partial^{|u|} \gtef}{\partial \bo{\x}_{u}}\left(\bo{\x}\right) 
\frac{\partial^{|u|} \gtef^\T}{\partial \bo{\x}_{u}}\left(\bo{\x}\right) 
 \int _{\Omega_u} \prod_{k \in u}
\frac{F_{k}\left[\min(\x_{k},\,\x_{k}')\right] -F_{k}(\x_{k})F_{k}(\x_{k}')}{\rho_{k} (\x_{k})^2} d\mu(\bo{\x}_u') d\mu(\bo{\x})
\nonumber \\
 &= &
\int 
\frac{\partial^{|u|} \gtef}{\partial \bo{\x}_{u}}\left(\bo{\x}\right) 
\frac{\partial^{|u|} \gtef^\T}{\partial \bo{\x}_{u}}\left(\bo{\x}\right) 
\left[\prod_{k \in u} \int _{\Omega_k}
\frac{F_{k}\left[\min(\x_{k},\,\x_{k}')\right] -F_{k}(\x_{k})F_{k}(\x_{k}')}{\rho_{k} (\x_{k})^2} d\mu(\bo{\x}_k')  \right] d\mu(\bo{\x})
\nonumber \\
 &= & \frac{1}{2^{|u|}} 
\int 
\frac{\partial^{|u|} \gtef}{\partial \bo{\x}_{u}}\left(\bo{\x}\right) 
\frac{\partial^{|u|} \gtef^\T}{\partial \bo{\x}_{u}}\left(\bo{\x}\right) 
\prod_{k \in u} \frac{F_{k}(\x_{k})\left[1-F_{k}(\x_{k})\right]}{\rho_{k} (\x_{k})^2}  d\mu(\bo{\x}) \, \nonumber ,   
\end{eqnarray}  
as           
\begin{eqnarray}  
\int _{\Omega_k}
F_{k}\left[\min(\x_{k},\,\x_{k}')\right] -F_{k}(\x_{k})F_{k}(\x_{k}') d\mu(\bo{\x}_k') &=& \int_{-\infty}^{x_k} F_{k}(\x_{k}') \left[1 -F_{k}(\x_{k})\right] \rho(x_k') \, dx_k'\nonumber \\ 
   & & + \int_{x_k}^{+\infty} F_{k}(\x_{k}) \left[1 -F_{k}(\x_{k}')\right]  \rho(x_k') \, dx_k' \nonumber \\ 
	 &=& \frac{1}{2}\left[1 -F_{k}(\x_{k})\right]F_{k}(\x_{k})^2 \nonumber \\ 
	 && + F_{k}(\x_{k})\left(\frac{1}{2} - F_{k}(\x_{k}) + \frac{1}{2} F_{k}(\x_{k})^2\right) \nonumber \\       
	&=&    \frac{1}{2} F_{k}(\x_{k}) \left[1 -F_{k}(\x_{k})\right] \, . 
	\nonumber           
\end{eqnarray}     

 The maxtrix $W$ is also a  symmetric and positive semi-definite matrix, as $F_{k}(\x_{k})\left[1-F_{k}(\x_{k})\right] \geq~0$.   \\       
  
Likewise, let define a covariance matrix $R$ like  
\begin{eqnarray} 
R & := & \int \left(\bo{g} -\bo{e} \right) \left(\bo{g}-\bo{e} \right)^\T \,  d\mu(\bo{\x})d\mu(\bo{\x}') \nonumber \\  
 & = & \int \bo{g} \bo{g}^\T \, d\mu(\bo{\x})d\mu(\bo{\x}')  -
\int \bo{g} \bo{e}^\T \, d\mu(\bo{\x})d\mu(\bo{\x}')   -
\int \bo{e} \bo{g}^\T \, d\mu(\bo{\x})d\mu(\bo{\x}')   +
\int \bo{e} \bo{e}^\T \, d\mu(\bo{\x})d\mu(\bo{\x}')  \, , \nonumber     
\end{eqnarray}        
with the functions       
$$    
\bo{g}\left(\bo{\x},\, \bo{\x}'_u \right) := \frac{\partial^{|u|} \gtef}{\partial \bo{\x}_{u}}\left(\bo{\x}\right) 
\prod_{k \in u}
\frac{\sqrt{F_{k}\left[\min(\x_{k},\,\x_{k}')\right] -F_{k}(\x_{k})F_{k}(\x_{k}')}}{\rho_{k} (\x_{k})}\, ,
$$
and             
$$ 
\bo{e}\left(\bo{\x}'_u,\, \bo{\x}\right) := \frac{\partial^{|u|} \gtef}{\partial \bo{\x}_{u}}\left(\bo{\x}'_u,\, \bo{\x}_{\sim u}\right) 
\prod_{k \in u}
\frac{\sqrt{F_{k}\left[\min(\x_{k},\,\x_{k}')\right] -F_{k}(\x_{k})F_{k}(\x_{k}')}}{\rho_{k} (\x_{k}')}\, .
$$             
Thirdly, we can remark that $W=\int \bo{g} \bo{g}^\T \, d\mu(\bo{\x})d\mu(\bo{\x}')=\int \bo{e} \bo{e}^\T \, d\mu(\bo{\x})d\mu(\bo{\x}')$, and
$ \var(\gtef) =\int \bo{g} \bo{e}^\T \, d\mu(\bo{\x})d\mu(\bo{\x}') =
\int \bo{e} \bo{g}^\T \, d\mu(\bo{\x})d\mu(\bo{\x}')$. Therefore
$       
R =2W -2\var(\gtef)         
$, or equivalently $ R/2= W-\var(\gtef)$. The Loewner ordering $\var(\gtef)  \matleq W$  follows, as $ W-\var(\gtef)=R/2$ is a symmetric and positive semi-definite matrix.  \\    
Finally, we have the equality in Corollary \ref{coro:equa}. 

\section{Proof of Theorem \ref{theo:pweqineq}} \label{app:pweqineq}
For Point (i), knowing that $\var(\M_u) \matleq \var(\M_u^{sup})$ and using Equation (\ref{eq:vardec}), we can write 
\begin{eqnarray}  
\var(\M) &\matleq & \sum_{\substack{u, \, u \subseteq \{1,\, \ldots,\, d\}\\ |u|>0}} \var(\tefg_u)     \nonumber \\
     &\matleq & \sum_{\substack{u, \, u \subseteq \{1,\, \ldots,\, d\}\\ |u|>0}} 
		\frac{1}{2^{|u|}} \int_{\Omega} 
\frac{\partial \M}{\partial \bo{\x}_u}\left(\bo{\x}\right)
\frac{\partial \M^\T}{\partial \bo{\x}_u}\left(\bo{\x}\right) 
 \prod_{k \in u} \frac{F_{k}(\x_{k})\left(1 -F_{k}(\x_{k})\right)}{\left[\rho_{k} (\x_{k})\right]^2}  \, d\mu(\bo{\x}) \nonumber \, , 
\end{eqnarray}     
bearing in mind Equation (\ref{eq:Intief}). \\

For Point (ii), $\forall \, j \in J_0$  consider a function $g_j : \R \to \R^\NN$  $g_j(x_j) =\M(x_j\, \bo{x}_{\sim j})$. Bearing in mind the Fubini-Lebesgue theorem, we can see that $\int_{\Omega_j} g_j \, d\mu_j =\bo{0}$ implies $\int_{\Omega} \M \, d\mu  =\bo{0}$ and $\var(\M) =\int_{\Omega_{\sim j}}\var(g_j)\, d\mu_{\sim j}$. By applying the optimal weighted Poincar\'e inequality  (\ref{eq:wpm1}) to $g_j$, we obtain
  
\begin{eqnarray}  
\var(\M) &\matleq & \frac{1}{2} \int_{\Omega} 
\frac{\partial \M}{\partial \x_{j}}\left(\bo{\x}\right) 
\frac{\partial \M^\T}{\partial \x_{j}}\left(\bo{\x}\right) 
 \frac{F_{j}(\x_{j})\left(1 -F_{j}(\x_{j})\right)}{\left[\rho_{j} (\x_{j})\right]^2}  \, d\mu(\bo{\x}) \, , \nonumber
\end{eqnarray} 
as $\frac{\partial \M}{\partial \x_{j}} =\frac{\partial g_j}{\partial \x_{j}}$. The result holds by taking the minimum value of the $|J_0|$ upper bounds of $\var(\M)$, with $|J_0|$ the cardinal of $J_0$.\\   
 
For Point (iii), first, we know from Point (i) that  
\begin{eqnarray}  
\var(\M) &\matleq & \sum_{\substack{u, \, u \subseteq \{1,\, \ldots,\, d\}\\ |u|>0}} \var(\tefg_u)     \nonumber \, .
\end{eqnarray}    

Second, let $D_j :=\int_\Omega \frac{\partial \M}{\partial \x_{j}}\left(\bo{\x}\right)
\frac{\partial \M^\T}{\partial \x_{j}}\left(\bo{\x}\right) 
 \frac{F_{j}(\x_{j})\left(1 -F_{j}(\x_{j})\right)}{\left[\rho_{j} (\x_{j})\right]^2}  \, d\mu(\bo{\x}) $ and $D_{(j)}$ be the $j^{\text{th}}$ smallest value of $D_j,\, j\in \{1,\,\ldots,\, d\}$. The integer $(j)$ is associated with the $j^{\text{th}}$ smallest value of $D_j,\, j\in \{1,\,\ldots,\, d\}$. Knowing that $\var(\tefg_u) \matleq \var(\tefg_j)$ when $\j \in u$ (\cite{roustant14}), we are going to replace   $\var(\tefg_u)$ with $\var(\tefg_j)$ when $j \in u$. In $\sum_{\substack{u, \, u \subseteq \{1,\, \ldots,\, d\}\\ |u|>0}} \var(\tefg_u)$ we have \\  
 $|\mathcal{A}_{(1)}|$ terms that contain $(1)$ (i.e.,  $(1) \in u$); \\ 
$|\mathcal{A}_{(2)}\backslash \mathcal{A}_{(1)} |$ terms such that $(2) \in u$  and $(1) \notin u$; \\     
$|\mathcal{A}_{(j)}\backslash \cup_{i \in \{(1),\, \ldots,\, (j-1)\} } \mathcal{A}_{(i)} |$ terms such that $(j) \in u$  and $(1),\, \ldots,\, (j-1) \notin u$.  Therefore, we have         
\begin{eqnarray}  
\var(\M) &\matleq & \sum_{j=1}^d  |\mathcal{A}_{(j)}\backslash \cup_{i \in \{(1),\, \ldots,\, (j-1)\} } \mathcal{A}_{(i)} | \var(\tefg_{(j)}) \nonumber \\ 
 &\matleq & \sum_{j=1}^d \frac{|\mathcal{A}_{(j)}\backslash \cup_{i \in \{(1),\, \ldots,\, (j-1)\} } \mathcal{A}_{(i)} |}{2}
 \int_{\Omega}    
\frac{\partial \M}{\partial \x_{(j)}}\left(\bo{\x}\right) 
\frac{\partial \M^\T}{\partial \x_{(j)}}\left(\bo{\x}\right) 
 \frac{F_{(j)}(\x_{(j)})\left(1 -F_{(j)}(\x_{(j)})\right)}{\left[\rho_{(j)} (\x_{(j)})\right]^2}  \, d\mu(\bo{\x})   \nonumber \, , 
\end{eqnarray}     
and Point (iii) holds.

\end{document}